\documentclass[10pt,a4paper]{amsart}
\usepackage[utf8]{inputenc}
\usepackage[T1]{fontenc}
\usepackage{amsmath}
\newcommand\numberthis{\addtocounter{equation}{1}\tag{\theequation}} % align* with a numbered line
\usepackage{amsfonts}
\usepackage{amssymb}
\usepackage{enumitem}
\usepackage[initials]{amsrefs}
\usepackage{xcolor}
\definecolor{darkred}{RGB}{180,0,0}
\definecolor{darkblue}{RGB}{0,0,180}
\usepackage{hyperref}
\hypersetup{
  colorlinks = true,
  urlcolor   = darkblue,
  linkcolor  = darkblue,
  citecolor  = darkred,
  pdftitle={Warped cones, (non-)rigidity, and piecewise properties by Damian Sawicki}
}
\usepackage[capitalise]{cleveref}

\newtheorem{thm}{Theorem}[section]
\newtheorem{theorem}[thm]{Theorem}
\newtheorem{thmx}{Theorem}
\newtheorem{questx}[thmx]{Question}
\newtheorem*{fact*}{Fact}
\newtheorem*{thm*}{Theorem}
\newtheorem*{question*}{Question}
\newtheorem{lem}[thm]{Lemma}
\newtheorem{rem}[thm]{Remark}
\newtheorem{prop}[thm]{Proposition}
\newtheorem{cor}[thm]{Corollary}
\newtheorem{ex}[thm]{Example}
\newtheorem{question}[thm]{Question}

\theoremstyle{definition}
\newtheorem{defi}[thm]{Definition}

\newtheorem{problem}[thm]{Problem}

\def\R{{\mathbb R}}
\def\C{{\mathbb C}}
\def\N{{\mathbb N}}
\def\Z{{\mathbb Z}}
\def\Q{{\mathbb Q}}
\def\F{{\mathbb F}}
\def\T{{\mathbb T}}
\def\SS{{\mathbb S}}

\def\U{{\mathcal U}}

\def\cF{{\mathcal F}}
\def\cO{{\mathcal O}}

\def\B{{\mathcal B}}

\def\cX{{\mathcal X}}
\def\cY{{\mathcal Y}}

\def\eps{\varepsilon}

\def\st{\ |\ }

\DeclareMathOperator{\id}{id}
\DeclareMathOperator{\im}{im}
\DeclareMathOperator{\supp}{supp}

\DeclareMathOperator{\diam}{diam}
\DeclareMathOperator{\asdim}{asdim}
\DeclareMathOperator{\asdimAN}{asdim_{AN}}
\DeclareMathOperator{\dimAN}{dim_{AN}}
\DeclareMathOperator{\eqasdim}{eq-asdim}

\DeclareMathOperator{\Prob}{Prob}

\newcommand{\ceil}[1]{{\left\lceil #1 \right\rceil}}
\newcommand{\floor}[1]{{\left\lfloor #1 \right\rfloor}}
\newcommand{\acts}{\ensuremath\raisebox{.6pt}{$\curvearrowright$}}
\newcommand*{\defeq}{\mathrel{\vcenter{\baselineskip0.5ex \lineskiplimit0pt
                     \hbox{\scriptsize.}\hbox{\scriptsize.}}}%
                     =}
\newcommand*{\eqdef}{=\mathrel{\vcenter{\baselineskip0.5ex \lineskiplimit0pt
                     \hbox{\scriptsize.}\hbox{\scriptsize.}}}}

\usepackage{mathrsfs}

\DefineSimpleKey{bib}{myurl}

\newcommand\myurl[1]{\url{#1}}

\BibSpec{webpage}{%
  +{}{\PrintAuthors} {author}
  +{,}{ \textit} {title}
  +{,}{ \myurl} {myurl}
  +{,}{ } {note}
  +{.}{ } {transition}
}

\BibSpec{webpage-double}{%
  +{}{\PrintAuthors} {author}
  +{,}{ \textit} {title}
  +{,}{ \textit} {titlea} 
  +{}{ and \textit} {titleb}
  +{,}{ \myurl} {myurl}
  +{,}{ } {note}
  +{.}{ } {transition}
}

\title{Warped cones, (non-)rigidity, and~piecewise~properties}
\dedicatory{\rm With an appendix by Dawid Kielak and Damian Sawicki}

\address{Dawid Kielak \hfill \texttt{dkielak@math.uni-bielefeld.de} \newline
Fakult\"at f\"ur Mathematik  \newline
Universit\"at Bielefeld \newline
Postfach 100131  \newline
D-33501 Bielefeld \newline
Germany}

\author{Damian Sawicki}
\address{Damian Sawicki \hfill {\url{www.impan.pl/~dsawicki/}  \newline
Institute of Mathematics \newline
Polish Academy of Sciences \newline
\'Sniadeckich 8 \newline
00-656 Warszawa \newline
Poland}}

\begin{document}
\begin{abstract}
We prove that if a quasi-isometry of warped cones is induced by a map between the base spaces of the cones, the actions must be conjugate by this map. The converse is false in general, conjugacy of actions is not sufficient for quasi-isometry of the respective warped cones. For a general quasi-isometry of warped cones, using the asymptotically faithful covering constructed in a previous work with Jianchao Wu, we deduce that the two groups are quasi-isometric after taking Cartesian products with suitable powers of the integers.

Secondly, we characterise geometric properties of a group (coarse embeddability into Banach spaces, asymptotic dimension, property A) by properties of the warped cone over an action of this group. These results apply to arbitrary asymptotically faithful coverings, in particular to box spaces. As an application, we calculate the asymptotic dimension of a warped cone, improve bounds by Szabó, Wu, and Zacharias and by Bartels on the amenability dimension of actions of virtually nilpotent groups, and give a partial answer to a question of Willett about dynamic asymptotic dimension.

In the appendix, we justify optimality of the aforementioned result on general quasi-isome\-tries by showing that quasi-isometric warped cones need not come from quasi-isometric groups, contrary to the case of box spaces.
\end{abstract}
\maketitle

\section{Introduction}\label{intro}

Finitely generated groups provide interesting examples of metric spaces to study by geometers. In particular, the first explicit construction of expanders, due to Margulis, consisted in taking a family of finite quotients of a group with relative property (T). Since then, such families have attracted a great deal of attention and are now known under the name of \emph{box spaces} \cite{Roe}.

One of the main ideas of geometric group theory is to study groups via their actions. For a group action on a compact space, J.\ Roe introduced a construction called \emph{a warped cone} \cite{Roe-cones}. In fact, the warped cone construction generalises the box space construction \cite{completions}.

For box spaces, one can use the ambient group to prove some properties of the box space and vice versa. In particular, amenability of the group is equivalent to property A of its box space
\cite{Roe}
and the Haagerup property of the group is equivalent to admitting a fibred coarse embedding into a Hilbert space by the box space \cites{Roe, CWY, CWW}. Furthermore, every quasi-isometry of box spaces lifts to a quasi-isometry of groups \cite{KV}.

Since a warped cone $\cO_\Gamma Y$ comes from an action $\Gamma\acts Y$ of a group, the group itself is not the appropriate (sufficient) object to be compared with a warped cone (as opposed to the above case of box spaces). An appropriate space was constructed recently by Jianchao Wu and the author \cite{SW}. 
Its usefulness comes from the fact that there is an \emph{asymptotically faithful} (see Definition \ref{faithful definition}) quotient map from it to the warped cone, just like the sequence of quotient maps from the group onto the box space is asymptotically faithful.

\subsection*{Results} Motivated by asymptotically faithful maps we introduce \emph{piecewise} versions of properties of metric spaces (like hyperbolicity, asymptotic dimension, property A, or coarse embeddability into Banach spaces). On the theoretical side, it enables us to characterise certain properties of groups by piecewise versions of these properties for box spaces (Theorem \ref{box space vs group}) or warped cones (Proposition \ref{nice group vs cone}), which was a problem raised by Osajda \cite{residually finite non-exact}. The aforementioned characterisations from \cites{Roe, CWW, CWY} concerned `equivariant' properties of groups, that is, taking into account their group structure, while our characterisations  concern metric properties.

For brevity, we omit some assumptions in formulations of theorems below.  They are all valid for warped cones over isometric free actions on nice spaces, including manifolds and ultrametric spaces (e.g.\ profinite completions metrised as in \cite{completions}).
\begin{thmx} Let $P$ be one of the following properties: `property A', `asymptotic dimension at most $k$', or `admitting a coarse embedding into $E$', where $k\in \N$ and $E$ is a Banach space. The following are equivalent:
\begin{itemize}
\item a group $\Gamma$ has property $P$;
\item a box space $(\Gamma/\Gamma_i)_i$ satisfies property $P$ piecewise;
\item a warped cone $\cO_\Gamma Y$ satisfies property $P$ piecewise.
\end{itemize}
\end{thmx}
On the quantitative side, we calculate the asymptotic dimension of the warped cone over an action $\Gamma\acts Y$, using the asymptotic dimension of $\Gamma$ and the topological dimension of $Y$ (Theorem \ref{main asdim}).
\begin{thmx} If $\asdim \cO_\Gamma Y < \infty$, then $\asdim \cO_\Gamma Y =  \asdim \Gamma\times \Z^m$, where~$m$ is the topological dimension of~$Y$.
\end{thmx}

Consequently, we have the following dynamical result (\cref{corollary for eqasdim}) about equivariant asymptotic dimension (also known as amenability dimension), which builds on and strengthens the estimates by Szabó, Wu, and Zacharias \cite{SWZ} and by Bartels \cite{coarse flow}. This addresses two problems of Willett, see \cref{section nuclear}.

\begin{thmx} If $\Gamma$ is a virtually nilpotent group and $Y$ is a compact metrisable space, then
 $\eqasdim(\Gamma \acts Y) \leq \asdim \Gamma + \dim Y$.

More generally, the inequality holds for any $\Gamma$ provided that $\eqasdim(\Gamma \acts Y)$ is finite.
\end{thmx}

The result that quasi-isometry of box spaces implies quasi-isometry of groups \cite{KV} was recently strengthened to the fact that it implies commensurability of groups \cite{DK fundamental}. It leads to the following question.

\begin{questx}\label{question rigidity}
Does a quasi-isometry $\cO_\Gamma Y\simeq \cO_\Delta X$ imply isomorphism / commensurability / quasi-isometry of $\Gamma$ and $\Delta$?
\end{questx}

In Appendix \ref{appendix}, we give a negative answer to this question. For instance, there exist free actions $\Z^3\acts \T^1$ and $\Z^2\acts\T^2$ yielding quasi-isometric warped cones (after passing to a subsequence). In fact, there exist (continuum many) free actions of $\Z^d$ for any $d\in \N$ yielding warped cones quasi-isometric with ones over trivial actions. We also construct actions of the infinite dihedral group yielding warped cones quasi-isometric with the collection of spheres over all positive radii.

Nonetheless, quasi-isometry of warped cones implies the following `stable' quasi-isometry (Theorem \ref{qi}), which was very recently obtained independently by de~Laat and Vigolo \cite{dLV}.
\begin{thmx}\label{introduction stable QI} Quasi-isometry of warped cones $\cO_\Gamma M\simeq \cO_\Delta N$ implies the following quasi-isometry of groups $\Gamma\times \Z^m\simeq \Delta\times \Z^n$,
for manifolds $M,N$ of dimension~$m$ and~$n$ respectively.
\end{thmx}

Since the study of warped cones is the study of group actions, one can ask about the case when a quasi isometry $\cO_\Gamma Y\simeq \cO_\Delta X$ is induced by a map $f\colon Y\to X$. We have the following result (Theorem \ref{isomorphism of groups}), which yields a very strong positive answer to \cref{question rigidity}.
\begin{thmx}
If a homeomorphism of connected spaces $f\colon Y \to X$ induces a quasi-isometry $\cO_\Gamma Y\simeq \cO_\Delta X$, then in fact $\Gamma$ and $\Delta$ are isomorphic and the actions are conjugate by $f$.
\end{thmx}
In \cref{induced QI section} we also study more general cases when $f$ is not required to be a homeomorphism and spaces $Y$ and $X$ are not connected.

However, in relation to a question of Pansu, we offer an example showing that topological conjugacy of actions is not sufficient for quasi-isometry of warped cones (Theorem \ref{examples from Thiebout and Ana}). The example also demonstrates that---contrary to all the other coarse invariants studied so far---coarse embeddability of the warped cone allows no dynamical or group-theoretic characterisation.
\begin{thmx}
There exist a $\Gamma$-space $Y$ with two invariant metrics $d$ and $d'$, such that $\cO_\Gamma(Y,d)$ does not admit a coarse embedding into a Hilbert space and $\cO_\Gamma(Y,d')$ does.
\end{thmx}

\section{Definitions}\label{definitions}

Throughout this note, $\Gamma$ will be a discrete finitely generated group and $S$ will be a finite symmetric generating set for $\Gamma$.

\begin{defi}[J. Roe \cite{Roe-cones}]\label{warped} Let $(Y,d)$ be a compact metric space with an action of  $\Gamma = \langle S\rangle$. \emph{The warped cone} is the collection of metric spaces $\cO_\Gamma Y = (\{t\}\times Y)_{t > 0}$ (denoted $(tY)_{t >0}$ for brevity), where each set $tY$ is equipped with the largest metric $d_\Gamma$ such that:
\[d_\Gamma(tx,ty) \leq d(tx,ty) \quad\text{ and }\quad d_\Gamma(ty,s(ty)) \leq 1 \text{ for any } s\in S,\]
where $d(tx, ty) = td(x,y)$.

The family $(tY,d)_{t>0}$ will also sometimes be used. It is called \emph{the open cone} (or \emph{the infinite cone}) and denoted by $\cO Y$.
\end{defi}

Note that the original definition \cite{Roe-cones} gives a metric on the \emph{unified} warped cone $\cO_\Gamma^\mathrm{u} Y = (0,\infty)\times Y = \bigcup_{t>0} tY$. However, some properties make sense only for a family of levels $tY$, rather than the metric space $\cO_\Gamma^\mathrm{u} Y$ build from them \cites{Vigolo, superexp}. For some other properties, a property holds uniformly across the levels if and only if it holds for $\cO_\Gamma^\mathrm{u} Y$ \cite{completions}*{Lemma 4.1 and Proposition 5.2}, but working with levels simplifies notation a bit.

Some results in this paper can be formulated and proved for the unified warped cone as well. However, for example in Theorem \ref{introduction stable QI}, the result would be slightly weaker: instead of the quasi-isometry $\Gamma\times \Z^m \simeq \Delta\times\Z^n$ one would get $\Gamma\times \Z^{m+1} \simeq \Delta\times\Z^{n+1}$. 

\begin{defi}\label{ce qi defi} Let $\cX$ and $\cY$ be two families of metric spaces.
\begin{itemize}
\item 
They are \emph{quasi-isometric}, denoted $\cX\simeq \cY$, if there exists an index set $I$ with surjective maps $I\owns i \mapsto X_i\in \cX$ and $I\owns i \mapsto Y_i\in \cY$ and a family of functions $f_i\colon X_i \to Y_i$ which are quasi-isometries in a uniform way.

That is, there exists constants $C\geq 1$ and $A\geq 0$
such that for every $i\in I$ and every $x,x'\in X_i$ we have
\begin{equation}\label{QI inequality}
C^{-1} d(x,x') - A \leq d(f_i(x), f_i(x')) \leq C d(x,x') + A
\end{equation}
(this double inequality constitutes the definition of a \emph{$(C,A)$-quasi-isometric embedding}) and the (open) $A$-neighbourhood of $\im (f_i)$ equals $Y_i$.
\item
Similarly, \emph{$\cX$ embeds coarsely into $\cY$} if there exists an index set $I$ with a surjective map $I\owns i \mapsto X_i\in \cX$ and any map $I\owns i \mapsto Y_i\in \cY$ and a family of functions $f_i\colon X_i \to Y_i$ that are coarse embeddings in a uniform way.

That is, there exist two non-decreasing and converging to infinity functions $\rho_-, \rho_+\colon [0,\infty)\to [0,\infty)$ such that for every $i\in I$ and $x,x'\in X_i$:
\begin{equation}\label{CE inequality}
\rho_- \circ d(x,x') \leq d(f_i(x), f_i(x')) \leq \rho_+ \circ d(x,x').
\end{equation}
\item
For brevity, we will call such families of maps $f=(f_i)$ respectively \emph{a quasi-isometry} or \emph{a coarse embedding} and write $f\colon \cX \to \cY$.
\end{itemize}
\end{defi}

We will use the wording `$(C,A)$-quasi-isometry' to specify constants. By relaxing inequality \eqref{QI inequality} in the definition of quasi-isometry to inequality \eqref{CE inequality}, one gets the definition of \emph{a coarse equivalence}, but we will typically stick to quasi-isometries and coarse embeddings.

As we emphasised in \cref{intro}, in addition to coarse embeddings and quasi-isometries, we will make heavy use of maps that are \emph{asymptotically faithful}.

\begin{defi}[Willett--Yu \cite{expgirth1}]\label{faithful definition} Let $(I,\leq)$ be a linearly ordered set and let $(f_i\colon Y_i\to X_i)_{i\in I}$ be a family of surjective maps. It is called \emph{asymptotically faithful} if for every $R\in \N$ there exists $j\in I$ such that for every $i\geq j$ and every $y\in Y_i$ the restriction $f_i\colon B(y,R)\to B(f_i(y),R)$ is an isometry.
\end{defi}
Practically, we will only use $I=\N$ or $I=(0,\infty)$ with their standard orders. The respective definition for a single surjective map $f\colon Y \to X$ would say that for every $R$ there is a bounded set $B\subseteq X$ such that we have the above isometric condition for $y\notin f^{-1}(B)$. Throughout, all balls are closed.

\subsection*{Examples}
One family of examples of spaces with interesting asymptotically faithful coverings appearing in the literature are \emph{box spaces} $\{\Gamma/\Gamma_i\}_{i\in \N}$, where $f_i$ is the quotient map $\Gamma\to \Gamma/\Gamma_i$. Here, $\Gamma$ is a finitely generated residually finite group and $\Gamma_i$ are finite index normal subgroups of $\Gamma$ such that $\Gamma_i\supseteq \Gamma_{i+1}$ and $\bigcap_i \Gamma_i = \{e\}$. Typically, one defines a box space as \emph{a coarse disjoint union} $(\bigsqcup \Gamma/\Gamma_i,d)$, where $d$ restricted to every `component' $\Gamma/\Gamma_i$ is equal to the quotient metric on $\Gamma/\Gamma_i$ and $d(\Gamma/\Gamma_j, \bigsqcup_{i\neq j} \Gamma/\Gamma_i) \to \infty$ as $j\to \infty$ (all such metrics are coarsely equivalent). With this definition, we can say that the quotient map $f\colon \{ 2^i \st i\in \N \} \times \Gamma \to \bigsqcup_i \Gamma/\Gamma_i$ is asymptotically faithful.

However, there is no canonical way to define distances between different components of a box space, so one cannot talk canonically about quasi-isometries of box spaces. This is one of the reasons why the author prefers to define a box space as a family of metric spaces. Another reason is that for a box space $\bigsqcup_i \Gamma/\Gamma_i$ the double box space $\bigsqcup_i \Gamma/\Gamma_\ceil{i/2}$ is not coarsely equivalent to it (at least when the original box space does not have duplicates itself). This is no longer true for our definitions, the double box space $\{\Gamma/\Gamma_i\}_i\sqcup \{\Gamma/\Gamma_i\}_i$ is quasi-isometric with the original box space $\{\Gamma/\Gamma_i\}_i$.

Another family of examples is given by sequences of graphs $(G_i)$ with increasing girth (that is, length of the shortest simple loop) and maps $f_i\colon \widetilde G_i\to G_i$, where $\widetilde G_i$ is the universal cover of $G_i$ and $f_i$ is the covering map. Note that $\widetilde G_i$ is a tree, and all $\widetilde G_i$ are isometric if vertices of graphs $G_i$ have a fixed degree (independent of~$i$ and a particular vertex).

Any box space of a free group (with the graph structure coming from its standard generators) fits in both categories.

A new class of examples was found in \cite{SW}, where it was proved that a suitably metrised product $\Gamma \times tY$ provides an asymptotically faithful covering for a warped cone via the map $(\gamma, ty) \mapsto \gamma(ty)$ if and only if the defining action $\Gamma\acts Y$ is free.

\begin{thm}[Sawicki--Wu \cite{SW}]\label{SW covering} Let $\Gamma \acts (Y,d)$ be a continuous action. Consider the space $\Gamma\times tY$ with the largest metric $d_1$ such that
\[d_1((\gamma, ty), (\eta, ty)) \leq |\gamma\eta^{-1}| \text{ and } d_1((\gamma, ty), (\gamma, ty')) \leq t d(\gamma y, \gamma y').\] 
Then, $(tY, d_\Gamma)$ is the quotient space of $\Gamma\times tY$ under the quotient map $\pi(\gamma, ty) = \gamma(ty)$.

Moreover, the family of quotient maps $(\pi_t\colon (\Gamma\times tY,d_1) \to (tY, d_\Gamma))_{t > 0}$ is asymptotically faithful if and only if the action $\Gamma\acts Y$ is free.
\end{thm}

By abuse of notation, we will write $(\Gamma \times \cO Y, d_1)$ or briefly $\Gamma \times \cO Y$ to denote the family of Cartesian products $\left((\Gamma\times tY,d_1)\right)_{t>0}$. Note that, when the action $\Gamma\acts Y$ is isometric, $d_1$ is the $\ell_1$-product metric of the word metric on $\Gamma$ and of the metric on the infinite cone, hence the notation.

We will also often use the following results.

\begin{lem}[\cite{completions}*{Proposition 2.1}]\label{2.1} Suppose $\Gamma \acts (Y,d)$ is Lipschitz and $L \geq 1$ is the Lipschitz constant for the action of the generators $S$ of $\Gamma$. Let 
\[D_\Gamma(ty,ty') = \inf_{\gamma\in\Gamma}( |\gamma| + td(\gamma y, y')).\]
Then the following inequality holds for all $y,y'\in Y$:
\[d_\Gamma(ty,ty')\leq D_\Gamma(ty,ty') \leq L^{{d_\Gamma(ty,ty')}}\cdot d_\Gamma(ty,ty').\]
In particular, when the action is isometric, then $D=d$.
\end{lem}

Similarly, we have the following.

\begin{lem}\label{cone in product} Suppose $\Gamma \acts (Y,d)$ is Lipschitz and $L \geq 1$ is the Lipschitz constant for the action of the generators $S$ of $\Gamma$, then:
\[
d_1((e,ty), (e,ty')) \leq d(ty,ty') \leq L^{d_1((e,ty), (e,ty'))}\cdot d_1((e,ty), (e,ty')).
\]
Consequently, the family of maps $(\iota_t\colon tY\to \Gamma\times tY)_{t>0}$ given by $\iota_t(ty) = (e,ty)$ is a coarse embedding $\cO Y \to \Gamma\times \cO Y$.
\end{lem}

\begin{lem}[\cite{completions}*{Remark 3.1}]\label{3.1} Let $\Gamma \acts (Y,d)$ be a continuous action. Then
\[\lim_{t\to\infty} d(ty, ty') = \begin{cases}
\infty,                        & \text{if } \Gamma y \neq \Gamma y',\\
\min \{|\eta| \st y'=\eta y \},   & \text{otherwise.}
\end{cases}
\]
In the latter case, $d(ty, ty')$ equals the limit for sufficiently large $t$.
\end{lem}

\section{Quasi-isometry of cones implies stable quasi-isometry of groups}
Recall that by an observation of Khukhro and Valette \cite{KV}, quasi-isometric box spaces must come from quasi-isometric groups. For finitely presented groups, this was recently strengthened by Delabie and Khukhro \cite{DK fundamental}: if two box spaces are quasi-isometric, the ambient groups must share a finite index normal subgroup. In fact, roughly speaking, from the geometry of the box space $\bigsqcup \Gamma/\Gamma_i$ one can recover the normal subgroups $\Gamma_i$, which shows that box spaces are very rigid.

In \cref{appendix}, we show that this is no longer true for warped cones: there exist quasi-isometric warped cones over actions of non-quasi-isometric groups (moreover, on non-homeomorphic spaces). Nonetheless, we have the following result yielding `stable' quasi-isometries.

\begin{thm}\label{qi}
Let $\Gamma, \Delta$ be two finitely generated groups and $M, N$ be two compact Riemannian manifolds of dimension $m$ and $n$ respectively with free isometric actions $\Gamma\acts M$ and $\Delta\acts N$. Then, quasi-isometry of warped cones $\cO_\Gamma M \simeq \cO_\Delta N$ implies quasi-isometry of groups $\Gamma\times \Z^m \simeq \Delta\times\Z^n$.
\end{thm}
\begin{proof}
Take a sequence $(t_i)\subseteq (0,\infty)$ going to infinity and let $(\tau_i)$ be another sequence such that $(f_i\colon(t_i M,d_\Gamma)\to(\tau_i N, d_\Delta))_{i\in \N}$ is a $(C,A)$-quasi-isometry for some $C\geq 1$ and $A\geq 0$. Since the diameter of $(t_i M,d_\Gamma)$ goes to infinity, the diameter of $(\tau_i N, d_\Delta)$ also goes to infinity and hence the sequence $(\tau_i)$ goes to infinity. Let $R\in \N$ and let $i=i(R)$ be sufficiently large (to be specified later).

Let $\Gamma \times t_i M$ be equipped with the metric from \cref{SW covering}, which for isometric actions is given by the formula:
\[d_1((\gamma, t_i x), (\gamma', t_i x')) = |\gamma'\gamma^{-1}| + t_i d(x,x')\]
and similarly for $\Delta \times \tau_i N$. Note that the ball of radius $R$ in $d_1$ about $(\gamma, t_i x)$ (and similarly for $(\delta, \tau_i y)$) is contained in the product 
\[B(\gamma,R)\times B(t_i x,R) = B(\gamma,R)\times B(x, R/t_i),\]
where we identified $t_i M$ and $M$. For $t_i$ sufficiently large, the map $\pi_\Gamma$ given by:
\[(\Gamma \times t_i M,d_1) \owns (\gamma,t_i x) \mapsto \gamma (t_i x) \in (t_i M, d_\Gamma),\]
when restricted to the above product, is an isometry by Theorem \ref{SW covering}.

Now, since the map $f_i$ is a $(C,A)$-quasi-isometry, the image under $f_i$ of any ball $B(t_i x, \rho)$ for any radius $\rho>0$ in $(t_i M,d_\Gamma)$ is contained in the ball $B(f_i(t_i x), C\rho+A)$ and moreover the $A$-neighbourhood of this image contains $B(f_i(t_i x), \rho/C-2A)$. By increasing $i$ if necessary, we can assume that the map $\pi_\Delta\colon (\delta, \tau_i y)\mapsto \delta (\tau_i y)$ is an isometry on $B((e, f_i(t_i x)), CR+A)$. Hence, we can consider a local inverse $\sigma_\Delta$ of $\pi_\Delta$ and obtain a map $j_i=\sigma_\Delta \circ f_i \circ \pi_\Gamma$:
\[j_i\colon B((e, t_i x), R) \to B((e, f_i(t_i x)), CR+A),\]
which is a $(C,A)$-quasi-isometric embedding such that for any $\rho\leq R$ the $A$-neighbourhood of the image of the ball $B((e, t_i x), \rho)$ contains $B((e,f_i(t_i x)), \rho/C-2A)$.

Let $r>0$ be a positive number such that the exponential maps $\exp_x\colon T_x M\supseteq B(0,r) \to B(x,r)$ and $\exp_y\colon T_y N\supseteq B(0,r) \to B(y,r)$ are homeomorphisms and that all $\exp_x$, $\exp_y$ and their inverses are $L$-Lipschitz maps with $L$ not depending on $x\in M$ and $y\in N$. We can also assume that these maps preserve spheres, that is, $d(0,z) = d(\exp_x(0), \exp_x(z))$ and $d(0,z') = d(\exp_y(0), \exp_x(z'))$ for any $z\in B(0,r) \subseteq T_x M$ and $z'\in B(0,r) \subseteq T_y N$.

For $i$ sufficiently large, we can assume that $r>R/t_i$ and $r>(CR+A)/\tau_i$. Consider the composition:
\[ k_i = (\id_\Delta \times (\tau_i \circ \exp_{f_i(t_i x)}^{-1})) \circ j_i \circ (\id_\Gamma \times (\exp_x \circ t_i^{-1})),\]
where $t_i^{-1}$ denotes the multiplication  (homothety) operator in $T_x M$ and $\tau_i$ denotes the multiplication operator in $T_{f(t_i x)} N$
(we keep identifying points in $t_iM$ and $M$, and points in $\tau_i N$ and $N$). The map $k_i$ is well defined on the ball $B((e,0),R)$ inside $\Gamma\times T_x M$ (with the $\ell_1$-metric)---which we will identify with $\Gamma \times \R^m$---and is a $(CL^2,AL)$-quasi-isometric embedding into the ball $B((e,0), CR+A)$ inside $\Delta \times \R^n$. Moreover, for any $\rho\leq R$ the $AL$-neighbour\-hood of the image of the ball $B((e, 0), \rho)$ contains the ball $B((e,0), \rho/C-2A)$.

Now, we can pre-compose the map $k_i$ with the isometric embedding of $B((e,0),R) \cap (\Gamma \times \Z^m)$ into $B((e,0),R)\subseteq \Gamma\times \R^m$ and post-compose with a closest-point retraction of $B((e,0), CR+A)$ onto $B((e,0), CR+A) \cap (\Delta \times \Z^n)$. Let us keep the old notation for this new map and new balls. Note that now it is a $(C',A')$-quasi-isometric embedding for $C'=CL^2$ and $A' = 2\sqrt n + AL$ and for any $\rho\leq R$ the $A''$-neighbourhood of the image of $B((e,0),\rho) \subseteq \Gamma \times \Z^m$ contains $B((e,0), \rho/C-2A) \subseteq \Delta \times \Z^n$, where $A''=(CL^2\sqrt{m} + \sqrt{n} + 2AL)$.

Now, let $\U$ be any non-principal ultrafilter on $\N$. We define
\[k\colon \Gamma \times \Z^m \to \Delta \times \Z^n \text{ by } k(\gamma,z) =\lim_{R\to \U} k_{i(R)}(\gamma, z).\]
 First, observe that for a given $(\gamma,z)$ the value $k_i(\gamma,z)$ is defined for sufficiently large~$i$, so the limit makes sense. Furthermore, since $k_i(e,0) = (e,0)$ for all $i$ and all $k_i$ are $(C',A')$-quasi-isometric embeddings, we know that $k_i(\gamma, z)$ always belongs to the finite set $B((e,0),\, C'd((\gamma,z), (e,0)) + A')$, and hence the limit exists. Obviously, it is still a $(C',A')$-quasi-isometric embedding. Furthermore, since $A''$-neighbourhood of $k_i(B((e,0), \rho))$ contains $B((e,0), \rho/C-2A)$, the same holds for~$k$. Since $\rho$ is arbitrary, the map $k$ is a quasi-isometry.
\end{proof}

In fact we did not use in the proof the fact that the whole warped cones are quasi-isometric: it suffices that there is an unbounded sequence $(t_i)_i\subseteq \R$ such that $(t_i M, d_\Gamma)_i$ is quasi-isometric with $(\tau_i N, d_\Delta)_i$ for some (necessarily unbounded) sequence $(\tau_i)_i$.

\begin{rem}\label{box rem} Similarly, one can prove that if a box space $(\Gamma/\Gamma_i)_i$ is quasi-isometric to (a subsequence of) a warped cone $(\tau_i N, d_\Delta)_{i}$ for an action $\Delta\acts N$ as in \cref{qi}, then $\Gamma \simeq \Delta \times \Z^n$. In this case, a box space of a group $\Gamma$ cannot be quasi-isometric to a warped cone over its action, unless $\Gamma \simeq \Gamma \times \Z^n$.

However, this is restricted to actions on manifolds, because any box space $(\Gamma/\Gamma_i)_i$ may be realised as a sequence of levels $(\tau_i Y, d_\Gamma)_{i}$ for $Y = \varprojlim \Gamma/\Gamma_i$ \cite{completions}.
\end{rem}

By similar arguments, one obtains the following.

\begin{cor}\label{ce}
Let $\Gamma, \Delta$ be two finitely generated groups and $M, N$ be two compact Riemannian manifolds of dimension $m$ and $n$ respectively with free isometric actions $\Gamma\acts M$ and $\Delta\acts N$. Then, if the warped cone $\cO_\Gamma M$ embeds coarsely into $\cO_\Delta N$, then there is a coarse embedding of groups $\Gamma\times \Z^m \to \Delta\times\Z^n$.
\end{cor}

In the proof of Theorem \ref{qi} we used only neighbourhoods of one orbit of the action $\Gamma\acts M$. Hence, with a more careful treatment, one can relax the assumption of the freeness of $\Gamma\acts M$ to the existence of a free orbit. The same applies to \cref{box rem} and \cref{ce}.

Tim de Laat and Federico Vigolo \cite{dLV} very recently obtained independently a different strengthening of \cref{qi}, in which they assume only that the actions $\Gamma\acts M$ and $\Lambda\acts N$ are \emph{essentially} free.

\section{Quasi-isometry induced by one map}\label{induced QI section}

The lack of rigidity exhibited in \cref{appendix} and in the formulation of Theorem~\ref{qi} and Corollary~\ref{ce} comes from the fact that a quasi-isometry of warped cones need not respect the equivalence relation induced by the action (it may ignore the division into the $\Gamma$- and $\Delta$-coordinate and the $M$- and $N$-coordinate (or $\Z^m$- and $\Z^n$-coordinate), that one has in $\Gamma\times t_i M$ and $\Delta\times \tau_i N$ (or $\Gamma \times \Z^m$ and $\Delta\times \Z^n$). However, it cannot happen if the quasi-isometry is induced by a single continuous map $M\to N$, in which case we have the following rigidity behaviour. 

\begin{theorem}\label{isomorphism of groups} Let $Y,X$ be compact metric spaces and $Y$ be connected, let $f\colon Y\to X$ be a continuous map, and let $\Gamma\acts Y$ contain a free orbit and $\Delta \acts X$ be free. If $f$ induces a quasi-isometry of warped cones $\cO_\Gamma Y$ and $\cO_\Delta X$, then
\begin{itemize}
\item there is a short exact sequence $1\to F\to \Gamma \to \Delta' \to 1$, where $\Delta' \leq \Delta$, the stabiliser of $\im(f)$, is of finite index in $\Delta$, and $F$ is finite;
\item $f$ factorises as $f'\circ q$ via the quotient map $q\colon Y\to Y/F$ and $f'$ is a homeomorphism conjugating the actions $\Delta'\acts \im (f)$ and $\Gamma/F\acts Y/F$;
\item $\im (f)$ is one of $[\Delta:\Delta']$-many connected components of $X$ and $\Delta$ acts on them transitively.
\end{itemize}

In particular, if $f$ is a homeomorphism, then $\Gamma$ and $\Delta$ are isomorphic and the actions are conjugate by $f$.
\end{theorem}

We sometimes omit it for brevity, but recall that whenever we write about a warped cone $\cO_\Gamma Y$, the metric space $Y$ is implicitly assumed to be compact, the group $\Gamma$ has to be finitely generated, and the action $\Gamma\acts Y$ has to be continuous (that is, by homeomorphisms).

A reader familiar with coarse equivalences may note that the following proof remains valid if one assumes that $f$ induces only a coarse equivalence. We present the proof in the setting of quasi-isometries, because as soon as $Y$ and $X$ are geodesic spaces, such a coarse equivalence actually has to be a quasi-isometry.

\begin{proof}[Proof of \cref{isomorphism of groups}]
Recall from Lemma \ref{3.1} that $\lim_{t\to \infty} d_\Gamma(ty, ty')$ is infinite if $y$ and $y'$ lie in different orbits and equals 
\begin{equation}\label{limit distance}
\min \{|\eta| \st \eta\in \Gamma, \eta y = y'\}
\end{equation}
if $\Gamma y'= \Gamma y$. By the assumption, for every $t>0$ there is $\tau(t)$ such that
\begin{equation}\label{estimates QI}
C^{-1}d_\Gamma(ty, ty') - A \leq d_\Delta(\tau(t)f(y), \tau(t)f(y')) \leq C d_\Gamma(ty, ty') + A,
\end{equation}
where $C\geq 1$ and $A\geq 0$ are uniform.
We conclude that 
\[\lim_{\tau\to\infty} d_\Delta(\tau f(y), \tau f(y')) = \lim_{t\to\infty} d_\Delta(\tau(t)f(y), \tau(t)f(y')) < \infty \]
if and only if $\lim_t d_\Gamma(ty, ty')< \infty$
and hence $f$ preserves orbits, that is, 
\begin{equation}\label{orbits preserved}
f(\Gamma y) = \Delta f(y) \cap \im (f),
\end{equation}
and also
\[\Gamma y\neq \Gamma y' \implies \Delta f(y)\neq \Delta f(y').\]

Now, since the $\Delta$-action is free, for every $\gamma\in \Gamma$ and $y\in Y$ there is a unique $\delta(\gamma, y)$ such that $f(\gamma y) = \delta f(y)$.

Due to \eqref{limit distance} and \eqref{estimates QI}, we know that 
\begin{equation}\label{range of delta}
 |\delta(\gamma, y)| \leq C|\gamma| + A,
\end{equation}
in particular $\delta(\gamma, \cdot)$ is a function $Y\to B(e, C|\gamma| + A)$ into a finite set and hence its continuity follows from the property that inverse images $\{y\in Y \st \delta(\gamma, y) = \mathrm{pt}\}$  of points are closed. Consequently, the map $\delta\colon\Gamma\times Y \to \Delta$ is continuous.

Fix $y\in Y$. We will prove that $\delta(\Gamma, y)$ is coarsely dense in $\Delta$. Denote $x = f(y)$ and take any $\delta\in \Delta$. 
By the assumption, for every $n\in \N$ there is $x_n \in \im (f)$ such that $d_\Delta(n\delta x, nx_n)\leq A$. In particular, for any $m\leq n$ we have $d_\Delta(m\delta x, mx_n)\leq A$.  By compactness, a subsequence of $x_n$ has a limit $x_0$ in the closure of $\im (f)$ (equal to $\im (f)$ itself under our assumptions), which satisfies $d_\Delta(m\delta x, mx_0)\leq A$ for all $m\in \N$. In particular $\delta x = \delta_0 x_0$ for some $|\delta_0|\leq A$.
This means that $x$ and $x_0$ lie in the same $\Delta$-orbit, and hence, 
by \eqref{orbits preserved} and the fact that $x_0\in \im (f)$, there is some $\gamma\in \Gamma$ such that $f(\gamma y) = x_0$
and consequently $\delta x = \delta_0 x_0 = \delta_0 \delta(\gamma, y) x$, that is:
\begin{equation}\label{coarsely dense}
\delta = \delta_0 \delta(\gamma, y).
\end{equation}

Assume now that $\Gamma \acts Y$ has a free orbit and let $y_0\in Y$ be an element of such an orbit. Then, formula \eqref{limit distance} simplifies to $\lim_t d_\Gamma (ty_0, \gamma ty_0) = |\gamma|$, which, together with \eqref{estimates QI} yields 
\begin{equation}\label{delta does not squeeze}
|\delta(\gamma, y_0)| \geq C^{-1} |\gamma| - A.
\end{equation}

Now recall that we assumed $Y$ to be connected. A continuous function $\delta(\gamma, \cdot)$ from a connected space to a discrete space must be constant. By definition $\delta$ is a cocycle, that is, it satisfies:
\[\delta(\gamma_2, \gamma_1 y) \delta(\gamma_1, y) = \delta(\gamma_2\gamma_1, y),\]
but as there is no dependence on $y$, it simply means that $\delta(\gamma)\defeq \delta(\gamma, y)$ is a homomorphism. In particular, inequality \eqref{delta does not squeeze} means that the kernel $F$ of $\delta$ is finite, as contained in the ball $B(e,CA)$. Similarly, the existence of elements $\delta_0\in \Delta$ satisfying formula \eqref{coarsely dense} and with uniformly bounded length (e.g.\ $|\delta_0|\leq A$ as above) is equivalent to saying that the subgroup $\delta(\Gamma)$ is of finite index in $\Delta$.

Note that $\delta(\gamma,y)=e$ if and only if $f(\gamma y)=f(y)$, but, since $\delta$ does not depend on the second coordinate, we have $\gamma\in F=\ker \delta$ if and only if $f(\gamma y)=f(y)$ for some (equivalently: all) $y$. Hence, $f$ is well defined as a map $Y/F \to \im f$. For the same reason, $f$ restricted to any $\Gamma$-orbit in $Y/F$ is injective and, in fact, it is injective on $Y/F$. Indeed, consider some $y,y'$ with $\Gamma y \neq \Gamma y'$---then by \cref{3.1} the limit $\lim_t d(ty, ty')$ is infinite and by inequality \eqref{estimates QI} so is $\lim_t d(\tau(t) f(y), \tau(t) f(y'))$, in particular $f(y)\neq f(y')$.

Since $Y$ is compact, $f\colon Y/F \to \im (f)$ is a homeomorphism. For any $y\in Y$ and $\delta' \in \Delta$ such that $\delta' f(y) \in \im (f)$, we know by equality \eqref{orbits preserved} that there is $\gamma\in \Gamma$ such that $f(\gamma y) = \delta' f(y)$. But then $\delta(\gamma)= \delta'$, which shows that $\Delta'\subseteq \delta(\Gamma)$. 
The converse inclusion is
clear and 
we also see
that if $\delta \im (f) \neq \im (f)$, then $\delta \im (f)\cap \im (f) = \emptyset$. Since $\Delta'$ has finite index, say $k$ in $\Delta$, there exist $\delta_1, \ldots \delta_{k-1} \in \Delta\setminus \Delta'$ such that $\{e,\delta_1, \ldots, \delta_{k-1}\} \im (f) = \Delta \im (f)$. Hence
\[\Delta \im (f) \setminus \im (f) = \bigcup_{i=1}^{k-1} \delta_i \im (f)\] is a closed set, so $\im (f)$ is clopen in $\Delta \im (f)$, thus a connected component.

By compactness of $\Delta \im (f)$, if there exists $x\in X\setminus \Delta \im (f)$, then a similar reasoning to the one in the proof of Lemma \ref{3.1} shows that $\lim_{\tau\to \infty} d_\Gamma(\tau x, \tau \Delta \im (f)) = \infty$, which would contradict the fact that the $A$-neighbourhood of $\im (f)$ is the whole of $\tau(t) Y$. Hence $X=\Delta \im (f)$.

The proof is finished in the case when $f$ is continuous. If it is additionally a homeomorphism, then $F$ must be trivial, as $f$ factors through $Y/F$. By surjectivity of $f$, equality \eqref{orbits preserved}, and the freeness of the action of $\Delta$, the map $\delta(\cdot, y)$ must be surjective. Hence $\delta$ is an isomorphism.
\end{proof}

\begin{rem} In \cref{isomorphism of groups}, if $f$ is assumed to be injective, then, instead of a~free $\Gamma$-orbit, it suffices to assume faithfulness of $\Gamma\acts Y$.
\end{rem}
\begin{proof}
Note that the only place in the proof of \cref{isomorphism of groups} where we used a free orbit of $\Gamma$ is inequality \eqref{delta does not squeeze}, which was later used to conclude that $\delta$ has a finite kernel. However, if $f$ is injective, we can conclude that already from the faithfulness of $\Gamma \acts Y$. Indeed, for every $\gamma\neq e$, there is some $y_\gamma\in Y$ with $\gamma y_\gamma \neq y_\gamma$ and, consequently, $f(\gamma y_\gamma) \neq f(y_\gamma)$. In particular $\delta(\gamma) = \delta(\gamma, y_\gamma) \neq e$, so $\delta$ is a monomorphism.
\end{proof}

By a similar argument, one can obtain:
\begin{cor} Let $Y,X$ be compact metric spaces and $Y$ be connected, let $f\colon Y\to X$ be continuous,  and let $\Gamma\acts Y$ contain a free orbit and $\Delta \acts X$ be free. If $f$ induces a coarse embedding of warped cones $\cO_\Gamma Y \to \cO_\Delta X$, then $\Delta$ contains a subgroup isomorphic to a quotient of $\Gamma$ by a finite index subgroup.
\end{cor}

Without the connectedness assumption, one obtains a cocycle instead of a homomorphism.

\begin{cor}\label{qi cocycle} If both actions are free and a homeomorphism $f\colon Y\to X$ induces a quasi-isometry of the respective warped cones, one gets a continuous cocycle $\delta\colon \Gamma \times Y \to \Delta$ (which is bijective as a function $\Gamma\to \Delta$ for a fixed $y\in Y$) satisfying 
\begin{equation}\label{quasi-isometric cocycle}
C^{-1} |\gamma| - A \leq |\delta(\gamma, y)|\leq C|\gamma| + A
\end{equation}
for some $C\geq 1$ and $A\geq 0$. In particular $\Gamma$ and $\Delta$ are bijectively quasi-isometric.
If $f$ is only injective or surjective, we get the appropriate properties of the cocycle.
\end{cor}
\begin{proof}
The inequality \eqref{quasi-isometric cocycle} is just a combination of \eqref{range of delta} and \eqref{delta does not squeeze} (now \eqref{delta does not squeeze} holds for any $y$ as all orbits are free). Consequently, the map $\delta_y = \delta(\cdot, y)\colon \Gamma\to \Delta$ is a quasi-isometric embedding. Indeed,
\begin{align*}
d(\delta_y(\gamma), \delta_y(\eta))
&= \big|\delta_y(\gamma) \, (\delta_y(\eta))^{-1}\big| \\
&= \big|\delta(\gamma, y) \, (\delta(\eta, y))^{-1}\big| \\
&= |\delta(\gamma, y) \, \delta(\eta^{-1}, \eta y)| \\
&= |\delta(\gamma \eta^{-1}, \eta y) | \\
&\in \big[C^{-1} |\gamma\eta^{-1}| - A, \, C |\gamma\eta^{-1}| + A\big] \\
& = \big[C^{-1} d(\gamma,\eta) - A, \, C d(\gamma,\eta) + A\big].
\end{align*}

Note that $\delta_y$ is injective by freeness of the $\Gamma$-action and injectivity of $f$ and it is surjective by freeness of the $\Delta$-action, surjectivity of $f$, and \eqref{orbits preserved}. If we do not assume surjectivity of $f$, then still some $A$-neighbourhood of $\delta_y(\Gamma)$ is the whole of $\Delta$ by the same argument as in the proof of Theorem \ref{isomorphism of groups}, so $\delta_y$ remains a quasi-isometry.
\end{proof}

Two actions $\Gamma \acts Y$ and $\Delta \acts X$ are orbit equivalent if there is an `isomorphism' $Y\to X$ satisfying \eqref{orbits preserved}. Typically, one considers probability measure preserving actions, Borel isomorphisms, and \eqref{orbits preserved} is required up to measure zero \cite{Vaes}. Alternatively, one can require the `isomorphism' to be a homeomorphism and the cocycles $\delta\colon \Gamma\times Y \to \Delta$ and $\gamma\colon \Delta\times X\to \Gamma$ to be continuous, like in \cref{qi cocycle}. The latter requirements constitute the definition of a \emph{continuous orbit equivalence}, which in some cases implies conjugacy of actions \cite{XinLi}.

The continuity of $f$ leads to the continuity of $\delta$ and, if $Y$ is connected, to a homomorphism. However, in the general case we still get the following. 

\begin{cor}\label{qi cocycle non-continuous} If both actions are free and a (continuous / measurable) function $f\colon Y\to X$ induces a coarse embedding 
of the respective warped cones, one gets a (continuous / measurable) cocycle $\delta\colon \Gamma \times Y \to \Delta$, which satisfies:
\begin{equation*}
\rho_-(|\gamma|) \leq |\delta(\gamma, y)|\leq \rho_+|(\gamma|),
\end{equation*}
for some non-decreasing and unbounded functions $\rho_\pm\colon [0,\infty)\to [0,\infty)$. In particular, $\Gamma$ embeds coarsely into $\Delta$. The embedding is injective / surjective whenever $f$ is injective / surjective.
\end{cor}

If the map $f$ above is surjective, we get a surjective coarse embedding of groups, hence a quasi-isometry. However, by the argument from the proof of \cref{isomorphism of groups}, even without surjectivity one can deduce quasi-isometry of $\Gamma$ and $\Delta$ whenever
the induced map of warped cones
is a quasi-isometry and $\im (f)$ is a closed set (e.g.\ if $f$ is continuous).

\begin{rem} Assume that $Y,X$ admit invariant ergodic measures, the actions are free up to measure $0$, and $f$, inducing a quasi-isometry $\cO_\Gamma Y\simeq \cO_\Delta X$, is a measure space isomorphism. In this case the actions are orbit equivalent by \cref{qi cocycle non-continuous}, so we can use various orbit equivalence super-rigidity or cocycle super-rigidity theorems to conclude that the respective actions are conjugate by $f$ (in particular, that $\Gamma$ and $\Delta$ are isomorphic).
\end{rem}

As an example, let us recall the following super-rigidity result, which deals with the type of actions studied in \cites{NS, Vigolo}.

\begin{thm}[Ioana \cite{Io13}]
Let $\Gamma<G$ and $\Delta < D$ be countable dense subgroups of compact connected Lie groups $G,D$ with trivial centers and assume that the left translation action $\Gamma \acts G$ has spectral gap (which is often the case, see references in \cite{NS}). Then the actions $\Gamma\acts G$ and $\Delta \acts D$ are orbit equivalent if and only if they are conjugate. 
\end{thm}

Regarding super-rigidity, we refer the reader to the survey \cite{Vaes} for additional information and more references, among which one should perhaps mention the paper \cite{Popa} of Popa, which establishes a super-rigidity result for certain actions $\Gamma\acts Y$ with essentially no restriction on the other action $\Delta \acts X$.

We finish with a few examples illustrating the results of this section.

\begin{ex}\label{examples to non-rigidity}\mbox{}
\begin{enumerate}
\item\label{divide by finite} It is not difficult to observe that for an isometric action of a finite group $F\acts Y$, the quotient map $f\colon Y \to Y/F$ induces a quasi-isometry between $\cO_F Y$ and $\cO_{\{e\}}(Y/F)$ (in fact an almost isometry with the additive constant $\diam F$).

This shows that it is reasonable to consider continuous maps and not only homeomorphisms in the results in this section, and also that the group~$F$ is necessary in the statement of \cref{isomorphism of groups}.
\item\label{divide by finite subgroup} The above example can be easily generalised. Assume that there is a finite normal subgroup $F \lhd \Gamma$ whose action $F\acts Y$ is isometric. Then the quotient map $f\colon Y \to Y/F$ induces a quasi-isometry between $\cO_\Gamma Y$ and $\cO_{\Gamma/F}(Y/F)$.
\item\label{finite groups on themselves} In the above examples $\Delta$ is a quotient of $\Gamma$ (that is, $\Delta=\Delta'$) and $f$ is surjective. The simplest example showing that it does not need to be the case is when $\Gamma \lneq \Delta$ are finite groups acting on themselves by left multiplication (and $f$ is the inclusion).

In fact, all warped cones over such actions are quasi-isometric and any $f\colon \Gamma\to \Delta$ induces a quasi-isometry. In particular, the respective cocycle $\delta$ need not be a homomorphism.

\item Note that by the last bullet point in \cref{isomorphism of groups}, if $Y$ is connected, $\Delta \acts X$ is free, and we want $\Delta'$ to be a proper subgroup of $\Delta$ as above, then the space $X$ must be disconnected.

\label{dihedral groups} 
Another possibility is to drop the assumption of freeness of $\Delta \acts X$. 
Consider the dihedral group $D_{2\cdot k}$, that is, the group of isometries of a regular $k$-gon $K$ of edge lengths equal to $2$. The embedding $f$ of the interval $[-1,1]$ as any edge of $K$ yields a quasi-isometry of $\cO_{D_{2\cdot 1}} [-1,1]$ and $\cO_{D_{2\cdot k}} K$, where $D_{2\cdot 1} \defeq \{1,-1\}$ acts antipodally on $[-1,1]$. In fact, these are all quasi-isometric to $\cO_{\{e\}} [0,1]$, as $[0,1] \simeq K/D_{2\cdot k}$.
\end{enumerate}
\end{ex}
\begin{proof}
For item \eqref{divide by finite}, note that by Lemma \ref{2.1}, we have
\[d_F(ty, ty') = \min_{g\in F} |g| + td(gy, y'),\]
and the quotient metric, by definition, equals 
\[d([ty], [ty'])= \min_{g,g'\in F} td(gy, gy') = \min_{g\in F} td(gy, y').\]

The same estimates hold for item \eqref{divide by finite subgroup}, that is, the quotient map is a quasi-isometry with constants $C=1$ and $A=\diam F$. Indeed, the quotient map $(tY, d_\Gamma) \to (tY/F, d_{\Gamma/F})$ is obviously $1$-Lipschitz and it suffices to show that it decreases distances by at most $\diam F$.

Let $y, y'\in Y$ and let the distance $d_{\Gamma/F}([ty], [ty'])$ be realised by sequences $(y_i)_{i=0}^n\subseteq Y$ and $(s_i)_{i=1}^{n-1} \subseteq S$ with $[y_0]=[y]$ and $[y_n]=[y']$, that is:
\[d_{\Gamma/F}([ty], [ty']) = td([y_0],[y_1]) + \sum_{i=1}^{n-1} 1 +  td([s_iy_i], [y_{i+1}])\]
(see \cite{Roe-cones}*{Proposition 1.6}). (One should think about the sequence $[y_0]$, $[y_1]$, $[s_1y_1]$, $[y_2],\ldots, [s_{n-1}y_{n-1}]$, $[y_n]$---then $1$ comes from the definition of the warped metric as an upper bound for $d_{\Gamma/F}([ty_i], s_i[ty_i])$ and similarly $td([s_iy_i], [y_{i+1}])$ is an upper bound for $d_{\Gamma/F}([s_ity_i], [ty_{i+1}])$.)

We can assume that $y_0=y$, then that $d(y_0, y_1) = d([y_0], [y_1])$, and further, by an inductive argument, that $d(s_i y_i, y_{i+1}) = d([s_iy_i], [y_{i+1}])$. It may happen that $y_n \neq y'$, but we still have $d_\Gamma(y_n, y')\leq \diam F$ (as $[y_n]=[y']$). Consequently, by applying the definition of the warped metric and the triangle inequality to the sequence $y_0$, $y_1$, $s_1y_1$, $y_2, \ldots$, $s_{n-1}y_{n-1}$, $y_n$, $y'$, we conclude that $d_\Gamma(ty,ty')\leq d_{\Gamma/F}([ty], [ty']) + \diam F$.

Item~\eqref{finite groups on themselves} follows from item~\eqref{divide by finite}.

Item~\eqref{dihedral groups} follows from item~\eqref{divide by finite} as well.
\end{proof}

\section{Dependence on metric}\label{dependence section}

We already know that quasi-isometry of warped cones does not imply quasi-isometry of their defining groups or homeomorphism of the spaces acted upon, but if the quasi-isometry is induced by a continuous map $f$, then the groups must be virtually isomorphic and the virtual isomorphism commutes with $f$ (\cref{isomorphism of groups}). One can ask whether the converse holds.

\begin{question}
If two actions are similar, are the respective warped cones quasi-isometric?
\end{question}

This also turns out to be false, in the strongest possible sense.

\begin{prop} There exists an action $\Gamma\acts Y$ and two invariant metrics $d,d'$ on the topological space $Y$ such that the warped cones $\cO_\Gamma (Y,d)$ and $\cO_\Gamma (Y,d')$ are not coarsely equivalent.
\end{prop}

The above proposition is quite trivial, the way that it is stated, as one can consider the trivial action on some space $(Y,d')$ and find a sufficiently `dissimilar' or `wild' metric~$d$.

What we will actually prove is that there is an action $\Gamma\acts Y$ and invariant metrics $d,d'$ such that there is an isometric embedding $(f_t\colon (tY, d) \to (tY, d'))_{t>0}$ (showing similarity of the metrics) and $\cO_\Gamma (Y,d)$ does not admit a coarse embedding into a Hilbert space yet $\cO_\Gamma (Y,d')$ does. (In particular, there is no coarse embedding $\cO_\Gamma (Y,d)\to \cO_\Gamma (Y,d')$ or a quasi-isometry between them.)

This is especially surprising if one realises that many coarse invariants of warped cones---like asymptotic dimension (see \cref{WZ} below, which is a result of J.~Wu and J.\ Zacharias), property A \cites{Roe, SW}, fibred coarse embeddability \cites{Roe,SW}, and expansion \cite{Vigolo} or super-expansion \cite{superexp}, as well as various \emph{piecewise} properties that we introduce in these notes (see \namecref{piecewise section}s \labelcref{piecewise section}--\labelcref{penultimate section})---admit characterisations in terms of the group action.

In other words, it turns out that \emph{one cannot predict coarse embeddability of the warped cone by just looking at the group, action, and the metric space $Y$, but it is necessary to consider all these jointly,} because both metrics $d,d'$ are very similar, if we disregard the action, and the difference between $\cO_\Gamma (Y,d)$ and $\cO_\Gamma (Y,d')$ only comes from the interplay of the action and the metrics.

\begin{thm}\label{examples from Thiebout and Ana}
There exists a compact space $Y$ with an action of the free group $\F_3$ admitting two invariant metrics $d$ and $d'$ such that:
\begin{itemize}
\item both infinite cones $\cO (Y,d)$ and $\cO (Y,d')$ have asymptotic (Assouad--Nagata) dimension equal to $0$;
\item there is an isometric embedding $f\colon (Y,d) \to (Y,d')$;
\item both warped cones $\cO_{\F_3}(Y,d)$, $\cO_{\F_3}(Y,d')$ do not have property A;
\item both warped cones $\cO_{\F_3}(Y,d)$, $\cO_{\F_3}(Y,d')$ admit a fibred coarse embedding into a Hilbert space;
\end{itemize}
however:
\begin{itemize}
\item the warped cone $\cO_{\F_3}(Y,d)$ does not admit a coarse embedding into the Hilbert space and $\cO_{\F_3}(Y,d')$ does;
\item in particular, there is no coarse embedding  $\cO_{\F_3}(Y,d) \to \cO_{\F_3}(Y,d')$, so they cannot be quasi-isometric.
\end{itemize}
\end{thm}

The action that we construct to prove the above is the action on a profinite group $Y$ by a dense subgroup $\Gamma < Y$, in particular it is: minimal, free, and admitting an ergodic invariant probability measure. Our construction relies on examples of box spaces obtained recently by Delabie and Khukhro \cite{DK}. They construct a sequence of normal subgroups in the free group $\F_3 \trianglerighteq M_1 \trianglerighteq N_1 \trianglerighteq M_2 \trianglerighteq N_2 \trianglerighteq\,\cdots$ with trivial intersection such that the box space associated to the sequence $M_i$ does not embed coarsely into the Hilbert space and the box space associated to $N_i$ does.

To a chain of quotients $G_i=\Gamma/\Gamma_i$ as in the definition of a box space (Subsection Examples in \cref{definitions}), one can also associate a compact group $\widehat{\Gamma}((\Gamma_i))=\varprojlim G_i$, called a \emph{profinite completion} or the \emph{boundary of the coset tree} (vertices of the tree are cosets of $\Gamma_i$). As $\Gamma$ maps into each $G_i$, it maps into the inverse limit, and this map is an injection onto a dense subgroup. In particular, $\Gamma$ acts on $\widehat{\Gamma}((\Gamma_i))$ by left translations. The inverse limit can be realised as a subset of the product $\prod_i G_i$: 
\[\{(g_i)_{i\geq 0} \st g_{i-1}=q(g_i)\ \forall i\geq 1 \}\]
(where $q\colon \Gamma/\Gamma_i\to \Gamma/\Gamma_{i-1}$ is the quotient map)
and hence it inherits the $\ell_\infty$-product metric $d((g_i),(h_i)) = a_j$, where $j$ is the first index with $g_j\neq h_j$ and the sequence $a_j>0$ converges monotonically to $0$.

The author showed in \cite{completions} that the box space and the completion as above are closely related. If one constructs the warped cone over $\widehat{\Gamma}((\Gamma_i))$ then it has property A if and only if the box space $(G_i)_i$ does. Moreover, if the rate of convergence of $a_j$ is large enough, then the warped cone embeds coarsely into the Hilbert space if and only if the box space does. Furthermore, the box space embeds quasi-isometrically into the warped cone. However, it was not clear how important the choice of metric is to these results. Theorem \ref{examples from Thiebout and Ana} shows that it is indeed crucial.

Note also that during his lecture in Cambridge in February 2017 Pierre Pansu asked the following (see \cite{metric2011-cambridge} for the exact statement):

\begin{question}[Pansu] Is the warped cone $\cO_\Gamma\widehat{\Gamma}((\Gamma_i))$ quasi-isometric with the box space $(\Gamma/\Gamma_i)_i$?
\end{question}

Theorem \ref{examples from Thiebout and Ana} provides a rather negative answer to this question: while by \cite{completions} one may construct a coarse (in fact an almost isometric, that is, quasi-isometric with multiplicative constant $C=1$) embedding $(\Gamma/\Gamma_i) \to \cO_\Gamma \widehat{\Gamma}((\Gamma_i))$, this is conditional on a suitable choice of a metric on $\widehat{\Gamma}((\Gamma_i))$ and the coarse geometry of $\cO_\Gamma \widehat{\Gamma}((\Gamma_i))$ depends on this choice, in particular it is not always compatible with the geometry of the box space $(\Gamma/\Gamma_i)$.

\begin{defi}\label{asdimAN}
For a subset $U$ of some metric space and $R>0$, \emph{an $R$-component of $U$} is a class of the equivalence relation on $U$ spanned by identifying points at distance smaller than $R$.

For $k\in \N$, a family of metric spaces $(X_i)_{i\in I}$ has \emph{asymptotic dimension at most $k$}, denoted $\asdim (X_i) \leq k$, if for every number $R\in \N_+$ there exists another number $S(R)\in \N$ such that for every $i\in I$ there is a covering $U_0, \ldots U_k$ of $X_i$ such that diameters of $R$-components of sets $U_l$ are bounded by $S$. The minimal such $k$ is called the asymptotic dimension of $(X_i)$ (if there is no $k$ as above, we say that the dimension is infinite).

If one requires $S$ to be a linear function of $R$, one obtains the definition of \emph{asymptotic Assouad--Nagata dimension} and if additionally $R$ ranges over $(0,\infty)$ rather than $\N_+$ of \emph{Assouad--Nagata dimension}, which are denoted by $\asdimAN$ and $\dimAN$, respectively.
\end{defi}

In fact, one can additionally assume the coverings $U_0, \ldots U_k$ to be open.

\begin{proof}[Proof of Theorem \ref{examples from Thiebout and Ana}] Consider a sequence of finite index normal subgroups in the free group $\F_3\trianglerighteq M_1 \trianglerighteq N_1 \trianglerighteq M_2 \trianglerighteq N_2 \trianglerighteq\,\cdots$ such that the box space $(G_i)=(\F_3/M_i)$ does not embed coarsely into the Hilbert space and the box space $(H_i)=(\F_3/N_i)$ does, as obtained in \cite{DK}. After passing to a subsequence we can assume $|{G_i}|/|{G_{i-1}}|\leq |{H_i}|/|{H_{i-1}}|$. We also assume that the sequence $(a_i)$ decreases sufficiently fast for the results of \cite{completions} to apply (to both box spaces $(G_i)$ and $(H_i)$).

Clearly, the inverse limit $Y$ of the inverse system $G_1\leftarrow H_1 \leftarrow G_2\leftarrow H_2 \leftarrow \cdots$ is the same as the inverse limit of the subsystem of $G_i$'s or the subsystem of $H_i$'s. Let $d$ denote the metric on $Y$ induced from the embedding into $\prod G_i$ and $d'$ denote the metric on $Y$ induced from the embedding into $\prod H_i$. We denote the images of these embeddings by $G$ and $H$.

\emph{Asymptotic dimension.~~}
We will now verify that $(tY,d)_{t>0}$ and $(tY,d')_{t>0}$ have Assouad--Nagata dimension $0$. Let $R>0$ and $t>0$. In the case of dimension zero, we have no choice but to put $U_0=tY$. The space $(tY,d)$ is an ultrametric space, that is, its metric satisfies a strong form of the triangle inequality: $d(tx,tz)\leq \max(d(tx,ty),d(ty,tz))$ (the same holds for $(tY,d')$). Hence, $R$-components of $U_0=tY$ have diameter bounded by $R$, so we can put $S(R)=R$ in the definition of Assouad--Nagata dimension. (Such an $R$-component of $(tY,d)$ consists of all points $t(g_i)\in tY$ with prescribed values of coordinates $g_i$ for $i<j$, where $j$ is the least index such that $ta_j<R$.)

\emph{Isometric embedding.~~} Now, we will construct an isometric embedding $e\colon (Y,d)\to (Y,d')$ by defining injective maps $e_i\colon G_i\to H_i$ such that $q\circ e_i = e_{i-1}\circ q$, where $q$'s are the appropriate quotient maps, and hence the product map $\prod e_i\colon \prod G_i \to \prod H_i$ will map $G$ into $H$.

Let $e_1$ be any injective map from $G_1$ to $H_1$ (we assume $G_0=H_0$ is the trivial group, so the assumption $|{G_i}|/|{G_{i-1}}|\leq |{H_i}|/|{H_{i-1}}|$ allows us to do that). Then, $e_2$ must map any element $g_2\in G_2$ into the set $q^{-1}(e_1(q(g_2)))$. This set has $|{H_2}|/|{H_1}|$ elements and there are $|{G_2}|/|{G_1}|$ elements in the set $q^{-1}(q(g_2))$ that must be mapped there. The existence of $e_2$ follows from our assumption on cardinalities and the existence of $e$ by induction. By construction $e$ is an isometry.

\emph{Property A and fibred embeddability.~~} Since $\F_3$ is non-amenable, lack of property A follows from \cite{Roe-cones}*{Proposition 4.1} (it is formulated in a less general setting, but the proof remains valid in our case; alternatively, one can use \cite{completions}*{Proposition 6.1}). The existence of a fibred coarse embedding was proved in \cite{SW}*{Theorem 3.2 together with Lemma 3.14}.

\emph{(Non-)embeddability.~~} Embeddability of a warped cone coming from an embeddable box space follows from \cite{completions}*{Theorem 7.4} and the analogous statement for non-embeddability follows from \cite{completions}*{Corollary 7.6}.
\end{proof}

\section{Asymptotically faithful coverings and piecewise properties}\label{piecewise section}

In this section, we introduce \emph{piecewise} versions of metric invariants, which---for spaces $X$ which cover asymptotically faithfully a sequence of spaces $Y_i$---enables us to characterise properties of $X$ by the piecewise versions of these properties for $(Y_i)$. The bulk of this section is devoted to proving simple lemmas establishing properties of this new notion.

Inspiration comes from the following recent question of Osajda.
\begin{question}[Osajda \cite{residually finite non-exact}]
What  are  coarse  geometric  properties  of  box  spaces  of  groups without property A?   Can  lack of property A  of  a  group  be  characterized  by  coarse
geometric properties of its box space?
\end{question}
Our result is more general than this question, as, apart from property A, it applies for instance to hyperbolicity, asymptotic dimension, or coarse embeddability into Banach spaces. Osajda explains in \cite{residually finite non-exact} that the answer to his original question is affirmative by unpublished results of Thibault Pillon via a new notion called \emph{fibred property A}. This name resembles \emph{fibred} coarse embeddability of \cite{CWY}, which is a much stronger condition than piecewise coarse embeddability introduced below.

While it is theoretically appealing that we can, in particular, characterise many geometric properties of groups by geometric properties of their box spaces (so far, one could do that for `group-theoretic' or `equivariant' properties like amenability \cite{Roe} or the Haagerup property \cites{Roe, CWY, CWW} and, more generally, property $P\B$ \cite{Arnt}), we are mostly interested in applications of this new notion. These are provided, in the case of asymptotic dimension, by a recent result of Yamauchi \cite{Yamauchi}, which we formulate at the end of this section using the general language developed here. Applications to warped cones are given in the next section.

\smallskip In what follows, we say that a collection of metric spaces, say a warped cone $(tY,d_\Gamma)_{t>0}$, has some property $P$ if all the spaces in the collection satisfy property $P$ with a uniform estimate on the respective constants. Examples of such {properties} include hyperbolicity (with a uniform constant $\delta$), coarse embeddability into a Banach space $E$ (with uniform estimates $\rho_-,\rho_+$ as in Definition \ref{ce qi defi}), asymptotic dimension at most $k$ (with a uniform function $R\mapsto S(R)$), or property A (with a uniform function $R\mapsto S(R)$, see Definition \ref{A defi}); a non-example would be `being a subset of $\R$' because this is not an isometry invariant or `countability of the collection' because it refers to the whole collection.

Such an approach was first introduced by Guentner, Tessera, and Yu in \cite{GTY}---where collections as above are called \emph{metric families}---in order to define \emph{finite decomposition complexity}.
For us this is mainly a matter of convenient language, as for example---for all properties that we will be interested in---a box space $\bigsqcup_i \Gamma/\Gamma_i$ has some property if and only if the property holds for the metric family $(\Gamma/\Gamma_i)_i$. This approach follows the spirit of \cref{ce qi defi}, extending the definition of a quasi-isometry and a coarse embedding from a single map between two metric spaces to families of maps between various spaces.

\smallskip Below, given a metric space $Y$, a scale $R<\infty$, and a forbidden subset $F \subseteq Y$, we consider the metric family $Y^R_{F+}$ of subsets of $Y$ with diameter at most $R$ and not intersecting $F$. Piecewise properties are defined by referring to properties of the family being the union of such $Y^R_{F+}$ over all scales $R<\infty$, where $F$ is a bounded subset allowed to vary with $R$. There is also the corresponding version for $Y$ being a metric family indexed by a linearly ordered set.

\begin{defi} Let $P$ be some property of metric families as discussed above.
\begin{itemize}
\item A metric space $Y$ \emph{has property $P$ piecewise} if for every $R\in \N$ there is a bounded subset $F(R)\subseteq Y$ such that the union $\bigcup_R Y_{F(R)+}^R$ of metric families satisfies property $P$, where
\[Y_{F+}^R \defeq \{ A \st A\subseteq Y\setminus F \text{ and } \diam A \leq R\}.\]
\item A family of metric spaces $(Y_i)_{i\in I}$ indexed by a linearly ordered set $(I,\leq)$ \emph{has property $P$ piecewise} if for every $R\in \N$ there is $j=j(R)\in I$ such that the union $\bigcup_R Y_{j(R)+}^R$ of metric families satisfies property $P$, where
\[Y_{j+}^R \defeq \{ A \st \exists i\geq j : A\subseteq Y_i \text{ and } \diam A \leq R\}.\]
(The definition also makes sense for $I$ being a directed set, but we will only consider the case $I=\N$ and $I=(0,\infty)$.)
\item Property $P$ is \emph{inherited by subsets} if for any metric family $\mathcal Y$ having property $P$, also the family $\mathcal Y_\supseteq = \{A \st \exists Y\in \mathcal Y : A\subseteq Y\}$ has property $P$.
\item Property $P$ is \emph{finitely determined} if for any proper and discrete metric space $X$, the following implication holds: if the metric family
\[X_{<\infty} = \{A \subseteq X \st \diam A < \infty\}\]
satisfies property $P$, then $X$ satisfies property $P$.
\item Property $P$ is \emph{coarsely invariant} if for any metric families $\mathcal Y$ and $\mathcal X$ and a coarse equivalence $f\colon\mathcal Y\to\mathcal X$, the family $\mathcal Y$ satisfies property $P$ if and only if $\mathcal X$ does.
\end{itemize}
\end{defi}

Clearly, property $P$ implies piecewise property $P$ under the assumption that $P$ is inherited by subsets.

It is well known that properties like property A, asymptotic dimension at most~$k$, or coarse embeddability into some Banach space $E$ are all inherited by subsets and coarsely invariant. Hyperbolicity (with a fixed constant $\delta$) is obviously finitely determined, as it can be defined using Gromov's `four point condition'. The fact that coarse embeddability into a fixed Banach space is finitely determined is the main result of \cite{Ostrovskii finite} (see also \cite{Baudier} for the case of $L^p$-spaces). Finite determination for property A and asymptotic dimension is verified in Proposition \ref{determination A asdim}.

Let us check some basic properties of our definitions.

\begin{lem} If $P$ is coarsely invariant and $(f_i\colon Y_i\to X_i)$ is a coarse equivalence, then $(Y_i)$ has property $P$ piecewise if and only if $(X_i)$ does.
\end{lem}
\begin{proof} It suffices to show that if $(X_i)$ has property $P$ piecewise, then $(Y_i)$ also does. So assume that $(X_i)$ has property $P$ piecewise.

Let $R\in \N$. For any $j\in I$ we can consider the `image' of $Y^R_{j+}$, that is, $f(Y^R_{j+}) \defeq \{f_i(A) \st i\geq j,\ A\subseteq Y_i \text{ and } \diam A\leq R\}$. Since $(f_i)$ is a coarse equivalence, there exists $S(R) \in \N$ (not depending on $j$) such that $\diam B \leq S(R)$ for all $B\in f(Y^R_{j+})$. Now, we know that there is $j(S(R))$ such that
\[\bigcup_{R\in \N}X_{j(S(R))+}^{S(R)}\]
satisfies property $P$ and hence the same holds for its subfamily $\bigcup_{R\in \N}f(Y^R_{j(S(R))+})$. Since $P$ is coarsely invariant, $\bigcup_{R\in \N} Y^R_{j(S(R))+}$ satisfies property $P$ as well, which finishes the proof.
\end{proof}

Clearly, the same can be said about quasi-isometrically invariant properties and quasi-isometric maps and likewise for rough isometries (that is, quasi-isometries with multiplicative constant $C=1$), also known as almost isometries.

Let us verify that piecewise properties for coarse disjoint unions and for sequences are equivalent (even if $P$ itself is not a coarse invariant).

\begin{lem} Let $(Y_i)$ be a sequence of bounded metric spaces and $P$ be any property. Then the sequence $(Y_i)$ has piecewise property $P$ if and only if the coarse disjoint union $Y=\bigsqcup_i Y_i$ has piecewise property $P$.
\end{lem}
\begin{proof}
Let us assume that $(Y_i)$ has piecewise property $P$, that is, for every $R\in \N$ there is $j(R)\in \N$ such that the union $\bigcup_R Y_{j(R)+}^R$ satisfies property $P$. Let $J(R)\geq j(R)$ be so large that $d(Y_k, Y_{k'}) > R$ inside $Y$ for $k \neq k' \geq J(R)$. Then, for $F(R)=\bigcup_{i=0}^{J(R)-1} Y_i$ the family $Y^R_{F(R)+}$ is a subfamily of $Y_{j(R)+}^R$, and hence the union $\bigcup_R Y_{F(R)+}^R$ has property $P$.

In the other direction, it suffices to take $j(R)$ so large that $F(R)\subseteq \bigcup_{i=0}^{j(R)-1} Y_i$ and then the family $Y_{j(R)+}^R$ is a subfamily of $Y^R_{F(R)+}$.
\end{proof}

The main motivation for the introduction of piecewise properties comes from the following two lemmas that lead to nice characterisations (see Theorem \ref{box space vs group} and \cref{nice group vs cone}) and, when combined with other results \cite{Yamauchi}, also to some unexpected equalities (see Corollary \ref{Yamauchi-box-spaces} and Theorem \ref{main asdim}).

\begin{lem}\label{global to piecewise} If $(f_i\colon X_i \to Y_i)$ is asymptotically faithful and $\mathcal X=(X_i)$ satisfies property $P$ which is inherited by subsets, then $(Y_i)$ has property $P$ piecewise.
\end{lem}
\begin{proof} Let $R\in\N$. By asymptotic faithfulness there exists $j(R)\in I$ such that the inverse image of any $A\in Y^R_{j(R)+}$ via the appropriate $f_i$ is a union of its isometric copies. Consider the family  $\U\defeq \bigcup_{R\in \N} f^{-1}(Y^R_{j(R)+})$ consisting of such inverse images of all $A\in \bigcup_{R\in\N} Y^R_{j(R)+} \eqdef \mathcal D$. Since $\U$ is a subfamily of $\mathcal X_\supseteq$ and property $P$ is inherited by subsets, the family $\U$ satisfies $P$. As elements of $\U$ consist of copies of elements of the family $\mathcal D$, the same holds for $\mathcal D$.
\end{proof}

Similarly, if $f\colon X\to Y$ is a single asymptotically faithful map and $X$ has property $P$ inherited by subsets, then $Y$ has property $P$ piecewise.

\begin{lem}\label{piecewise to global} Let $X$ be a discrete proper metric space, $P$ be finitely determined and $(f_i\colon X\to Y_i)$ be an asymptotically faithful covering such that $(Y_i)$ has property $P$ piecewise. Then $X$ has property $P$.
\end{lem}
\begin{proof}
Let $R\in \N$ and let $A\subseteq X$ have diameter bounded by $R$. By the asymptotic faithfulness there is $j(R)\in \N$ such that $f_i\colon B(x,R) \to B(f_i(x), R)$ is an isometry for any $x \in X$ and $i\geq j$. By piecewise property $P$ we may increase $j(R)$ so that the family 
\[\{f_{i}(A) \st R \in \N\text{, } i\geq j(R)\text{, } A \subseteq X \text{ and } \diam A\leq R \}\]
 (as a subfamily of $\bigcup_R Y_{j(R)+}^R$) satisfies property $P$. But the above family is isometric to
 \[\{ A \st A \subseteq X \text{ and } \diam A < \infty \} = X_{<\infty}\]
 and hence $X$ has property $P$ by finite determination.
\end{proof}

We will use the following definition of property A, equivalent to the original one \cite{A} for bounded geometry metric spaces \cite{Roe-cones} (see also \cite{SW}).

\begin{defi}\label{A defi} A metric space $X$ has \emph{property A} if for every $R\in \N$ there is $S\in \N$ and a map $f\colon X\to \Prob(X)$ such that $\|f(x) - f(x')\| \leq 1/R$ for $d(x,x')\leq R$ and $\supp f(x) \subseteq B(x,S)$.
\end{defi}

\begin{prop}\label{determination A asdim} Property A and having asymptotic dimension at most $k\in \N$ are finitely determined.
\end{prop}
\begin{proof} Let $X$ be a discrete proper metric space.

Assume first that $X_{<\infty}$ has asymptotic dimension at most $k$. That is, for any $R\in \N$ there is $S\in \N$ such that for any $A\subseteq X$ of finite diameter we have sets $U_0^A, \ldots, U_k^A$ such that $\bigcup_{l=0}^k U_l^A = A$ and $R$-components of $U_l^A$ have diameter bounded by $S$. We can assume that the sets $U_l^A$ are mutually disjoint, and hence the above partition of $A$ into the sets $U_l^A$ can be encoded by a function $f^A\colon A\to \{0,\ldots,k\}$.

Now, consider any non-principal ultrafilter $\U$ on $\N$ and let $f\colon X\to \{0,\ldots,k\}$ be the limit $f(x) = \lim_{n\to \U} f^{B(x_0,n)}(x)$ for some fixed $x_0\in X$. Consider an $R$-component of $f^{-1}(l)$. If its diameter is not $S$-bounded, there must be a pair of points $x,x'\in f^{-1}(l)$ with $d(x,x') > S$ and a finite sequence $x=x_0, x_1, \ldots x_m = x' \in f^{-1}(l)$ with $d(x_{i-1}, x_i) < R$. By properties of ultrafilters, there must be (infinitely many) $n\in \N$ such that $U_l^{B(x_0,n)}$ contains this sequence, and hence some $R$-component of $U_l^{B(x_0, n)}$ has diameter greater than $S$, a contradiction.

As for property A, assume that for every $R\in \N$ there is $S\in \N$ such that for any $A\subseteq X$ of finite diameter we have maps $A \owns x \mapsto f_x^A \in \Prob(A)$ such that $\|f_x^A - f^A_{x'}\| \leq 1/R$ for $x,x'\in A$ with $d(x,x')\leq R$ and that $\supp f_x^A \subseteq B(x,S)$. As before, for every $x\in X$ define $f_x \in \Prob(X)\subseteq \ell_1(X)$ by the formula 
\[f_x(y) = \lim_{n\to \U} f^{B(x_0, n)}_x(y),\]
where $y\in X$.
The limit exists by compactness of the interval $[0,1]$, we still have $\supp f_x \subseteq B(x,S)$, and $f_x$ remains a probability measure by finiteness of $B(x,S)$. Since the difference $f_x - f_{x'}$ is also supported on a finite set, the inequality $\|f_x - f_{x'}\| \leq 1/R$ for $d(x,x')\leq R$ is preserved in the limit as well.
\end{proof}

Hence, we have obtained the following characterisation of metric invariants of a group by metric invariants of its box space.

\begin{thm}\label{box space vs group} Let $P$ be `hyperbolicity',  `property A', `asymptotic dimension at most $k$', or `coarse embeddability into a Banach space $E$'. Let $\Gamma$ be a residually finite group with a box space $(\Gamma/\Gamma_i)$. Then, $\Gamma$ has property $P$ if and only if the box space $(\Gamma/\Gamma_i)$ has property $P$ piecewise.
\end{thm}

The following surprising results were recently obtained by Takamitsu Yamauchi \cite{Yamauchi}. They were formulated for large girth graphs and without using the language of piecewise properties, but Yamauchi's proof remains valid in this more general setting. It was observed by Rufus Willett and mentioned during his lectures in Southampton in March 2017 \cite{metric2011-southampton} (in the case of box spaces).

\begin{thm}[Yamauchi]\label{Yamauchi}
Assume that a family of bounded spaces $(X_i)_{i\in \N}$ has piecewise asymptotic dimension at most $k$ and that $\asdim (X_i)_i \leq n < \infty$. Then in fact $\asdim (X_i)_i \leq k$.
\end{thm}
\begin{proof} By the assumption, for every $R\in \N$ there is $S(R)\in \N$ such that every $X_i$ admits a partition into sets $U_{i,0}(R), \ldots, U_{i,n}(R)$ whose $R$-components are $S(R)$-bounded. Additionally, for every $S\in \N$ there is $j(S)\in \N$ such that the family of sets $X_{j(S)+}^S$ has asymptotic dimension at most $k$ (uniformly in $S$). We can assume that maps $R\mapsto S(R)$ and $S\mapsto j(S)$ are strictly increasing.

Let us cover every set $X_i$ by $n+1$ subsets $V_{i,0},\ldots, V_{i,n}$ as follows. For $i < j({S(1)})$ we put $V_{i,\iota} = X_i$ for every $0\leq \iota \leq n$. For $j({S(R)}) \leq i < j({S({R+1})})$ we put $V_{i,\iota} = U_{i,\iota}(R)$. Hence, for $j({S(R)}) \leq i < j({S({R+1})})$, $R$-components of $V_{i,\iota}$ are $S(R)$-bounded.

We will show that for every $\iota\in \{0,\ldots, n\}$ the family $(V_{i,\iota})_i$ has asymptotic dimension at most $k$. Fix one such $\iota$ and let $r\in \N$. For $i<j({S(r)})$ we can trivially decompose $V_{i,\iota}$
into sets $W_{i,\iota,0}=\ldots = W_{i,\iota,k} = V_{i,\iota}$ and clearly $r$-components of $W_{i,\iota,l}$ are $(\max_{i<j({S(r)})} \diam (X_i))$-bounded.
For $i\geq j(S(r))$, let $R\geq r$ be such that $j({S(R)}) \leq i < j({S({R+1})})$---then
every $R$-component of $V_{i,\iota}=U_{i,\iota}(R)$ is $S(R)$-bounded and in particular every $r$-component $C$ of $V_{i,\iota}$ is $S(R)$-bounded. Since $i\geq j({S(R)})$, by the assumption of piecewise asymptotic dimension at most~$k$, every such component $C$ admits a decomposition into sets $W_{i,\iota,0}^C, \ldots, W_{i,\iota,k}^C$ whose $r$-components are $T$-bounded (where $T$ depends on $r$, but not on $C$ or $R$). We put $W_{i,\iota,l}=\bigcup_C W_{i,\iota,l}^C$, where the union is over $r$-components of $V_{i,\iota}$. Since $W_{i,\iota,l}^C\subseteq C$, we conclude that $r$-components of the union are still $T$-bounded.

Since every $X_i$ is a finite union $\bigcup_{\iota=0}^n V_{i,\iota}$ (of the fixed length $n+1$) and for every $\iota$ the family $(V_{i,\iota})_i$ has asymptotic dimension at most $k$, by using the finite union theorem for asymptotic dimension \cite{Bell--Dranishnikov} we conclude that the family $(X_i)$ has asymptotic dimension at most $k$.
\end{proof}

\begin{cor}[Yamauchi]\label{Yamauchi-box-spaces} Let $(\Gamma/\Gamma_i)$ be a box space. Then $\asdim (\Gamma/\Gamma_i)$ is infinite or equals $\asdim \Gamma$.
\end{cor}
\begin{proof}
If $\asdim (\Gamma/\Gamma_i) = \infty$, there is nothing to prove, so assume that $\asdim (\Gamma/\Gamma_i)$ is finite. Let $k\leq \asdim (\Gamma/\Gamma_i)$ be the minimal number such that $(\Gamma/\Gamma_i)$ has asymptotic dimension at most $k$ piecewise. By Theorem \ref{box space vs group} $\asdim \Gamma = k$ and by Theorem \ref{Yamauchi} we have $\asdim (\Gamma/\Gamma_i) \leq k$; consequently $\asdim (\Gamma/\Gamma_i) = k = \asdim \Gamma$.
\end{proof}

\section{Piecewise properties and warped cones}\label{penultimate section}

In this section we will apply the definitions and results of the previous section to warped cones. The main result is as follows.

\begin{theorem}\label{main asdim}
Let $M$ be an $m$-dimensional Riemannian manifold and $\Gamma \acts M$ be free and isometric. Assume that $\asdim \cO_\Gamma M < \infty$. Then
\[\asdim \cO_\Gamma M = \asdim \Gamma \times \Z^m.\]
Let $Y$ be an ultrametric space (more generally: whenever $\asdim \cO Y = 0$) with a free Lipschitz action $\Gamma\acts Y$ and assume that $\asdim \cO_\Gamma Y < \infty$, then 
\[\asdim \cO_\Gamma Y = \asdim \Gamma.\]
\end{theorem}

Before proving this, let us state the following two propositions that gather results on piecewise properties of warped cones.

\begin{prop}\label{asdim for the product}
Let $\Gamma\acts Y$ be a Lipschitz action.
\begin{enumerate}
\item\label{1} $(\Gamma \times \cO Y,d_1)$ has property A if and only if both $\Gamma$ and $\cO Y$ have property A.
\item\label{2} $(\Gamma \times \cO Y,d_1)$ has finite asymptotic dimension if and only if both $\Gamma$ and $\cO Y$ have finite asymptotic dimension; in fact, 
\[\asdim\Gamma, \asdim \cO Y 
\leq \asdim (\Gamma \times \cO Y,d_1)\leq \asdim \Gamma + \asdim \cO Y.\]
\item\label{3} Assume that the action is isometric and $E$ is a Banach space such that $E\oplus E$ is isomorphic to a subspace of $E$. Then, clearly, $(\Gamma \times \cO Y,d_1)$ embeds coarsely into $E$ if and only if both $\Gamma$ and $\cO Y$ do.
\end{enumerate}
If the action is additionally free, and the equivalent conditions mentioned in \eqref{1}, \eqref{2}, or \eqref{3} hold, then the warped cone $\cO_\Gamma Y$ has, respectively, piecewise property A, piecewise finite asymptotic dimension, or piecewise embeds coarsely into $E$.
\end{prop}
\begin{proof} For the `only if' part of \eqref{1}, \eqref{2}, and \eqref{3} it suffices to observe that $\Gamma$ and $\cO Y$ embed coarsely into $(\Gamma \times \cO Y,d_1)$. For $\Gamma$, the map $\gamma\mapsto (\gamma, x)$ is an isometric embedding for any $t>0$ and $x\in tY$. For $\cO Y$, the family of maps (over $t>0$):
\[(tY, d) \owns x\mapsto (e,x) \in (\Gamma \times tY, d_1)\]
forms a coarse embedding by \cref{cone in product}.

The `if' part of assertion \eqref{1} is proved in \cite{cBCc} and the `if' part of assertion \eqref{3} is straightforward: it suffices to define coarse embeddings $f_t\colon \Gamma\times tY \to E$ as the sum of coarse embeddings $f_\Gamma \colon \Gamma \to E$ and $f_{tY} \colon  tY \to E$, namely 
\[f(\gamma, x) = f_\Gamma(\gamma)\oplus f_{tY}(x) \in E\oplus E\subseteq E\]
(recall that when the action is isometric, then $d_1$ is the product metric).

To verify the `if' part of assertion \eqref{2}, we will use the following Hurewicz-type theorem:
\begin{thm}[Brodskiy--Dydak--Levin--Mitra \cite{BDLM}]\label{BDLM} Let $f\colon X\to Z$ be a large-scale uniform map. Then $\asdim X \leq \asdim Z + \asdim f$, where 
\[\asdim f = \sup \{ \asdim f^{-1}(A) \st A\subseteq Z \text{ and } \asdim A = 0\}.\]
\end{thm}
Here, \emph{large-scale uniform} (or \emph{bornologous} in the terminology of J.~Roe \cite{Roe}) means that there exists a non-decreasing function $\rho_+$ such that
$d(f(x), f(x')) \leq \rho_+ \circ d(x,x')$. (Note that the theorem is formulated for a single map and we will apply it to a family of maps $f_t\colon X_t\to Z$, which is justified because we will show the assumptions to hold uniformly over $t>0$.)

For $t>0$ we define $f_t\colon (\Gamma\times t Y, d_1) \to \Gamma$ as the projection onto the first coordinate. Clearly, this is a $1$-Lipschitz map. Let $A \subseteq \Gamma$ have asymptotic dimension equal to $0$. That is, for every $R\in \N$ the diameters of $R$-components of $A$ are uniformly bounded by some constant $D$. We want to show that $\asdim ( f_t^{-1}(A) )_{t>0} \leq \asdim \cO Y \eqdef k$. By~definition we have $f_t^{-1}(A) = A\times t Y$, in particular for two different $R$-components $C, C'\subseteq A$ the distance of $f_t^{-1}(C)$ and $f_t^{-1}(C')$ is at least $R$. Consequently, in order to show that $f_t^{-1}(A)$ can be partitioned into $k+1$
sets whose $R$-components are uniformly bounded, it suffices to partition every $f_t^{-1}(C)$.

Hence, fix one $R$-component $C$ of $A$ and pick $\gamma\in C$. Using the isometric group action on $\Gamma\times \cO Y$ and the compatible isometric action on $\Gamma$ by right multiplication, we can assume $\gamma=e$. Let $U_0, \ldots U_k$ be a partition of $(t Y, d)$ whose $(L^{R+2D}(R+2D))$-components are $S$-bounded (note that $S$ can be chosen independently of~$t$). Define $W_0, \ldots W_k$ by $W_l = C\times U_l$.

Now if $d_1((\eta, x), (\delta, z)) < R$ for $\eta, \delta \in C$, then we have $d_1((e, x), (e, z)) < R + 2D$. By the latter part of the following double inequality from \cref{cone in product}
\begin{equation*}
d_1((e,ty), (e,ty')) \leq d(ty,ty') \leq L^{d_1((e,ty), (e,ty'))}\cdot d_1((e,ty), (e,ty')),
\end{equation*}
we conclude that $d(x,z) < L^{R+2D}(R+2D)$. Hence, by the former part of this inequality, we obtain that the projection onto $\{e\}\times t Y$ (given by $(\eta, x)\mapsto (e,x)$) of an $R$-component of $W_l$ is $S$-bounded. Consequently, any such component must be $(S+2D)$-bounded. This completes the verification of the assumptions of \cref{BDLM} and hence proves that $\asdim (\Gamma \times \cO Y,d_1) \leq \asdim \Gamma + \asdim \cO Y$.

The last part of the proposition is a direct application of Lemma \ref{global to piecewise} in the setting of \cref{SW covering}.
\end{proof}

In particular, for reasonable spaces $Y$ and free actions, we have the following equivalences.

\begin{prop}\label{nice group vs cone} Assume that $\cO Y$ has finite asymptotic dimension (which is the case if $Y$ is a manifold, finite simplicial complex, or a profinite completion) and the action $\Gamma\acts Y$ is free and Lipschitz. Then:
\begin{enumerate}
\item\label{next1} $\Gamma$ has property A if and only if $\cO_\Gamma Y$ has property A piecewise.
\item\label{next2} $\Gamma$ has finite asymptotic dimension if and only if $\cO_\Gamma Y$ has finite asymptotic dimension piecewise.
\item\label{next3} Assume that the action is isometric and let $E$ be a Banach space such that there exists an infinite dimensional space $B$ such that $E\oplus B$ is isomorphic to a subspace of $E$. Then, $\Gamma$ embeds coarsely into $E$ if and only if $\cO_\Gamma Y$ embeds coarsely into $E$ piecewise.
\end{enumerate}
\end{prop}

In fact, in the proof of the `if' parts of assertions \eqref{next1}, \eqref{next2}, and \eqref{next3} only one free orbit is used 
and the actions do not need to be isometric or Lipschitz. If $Y$ embeds bi-Lipschitz in $\R^n$, it suffices to assume $\dim B\geq n$.
\begin{proof}
The `only if' parts of \eqref{next1} and \eqref{next2} follow from Proposition \ref{asdim for the product} and the fact that finite asymptotic dimension (that we assume about $\cO Y$) implies property A.
Similarly for \eqref{next3}, finite asymptotic dimension implies coarse embeddability into a~Hilbert space \cite{A} and since levels of the infinite cone are compact spaces, we can embed them into finite-dimensional Hilbert spaces as well (with the dimension depending on the level!). As the family of Euclidean spaces $(\R^i)_i$ embeds bi-Lipschitz into any infinite dimensional Banach space $B$, we get that $\cO Y$ embeds coarsely into $B$. Thus whenever $\Gamma$ embeds coarsely into $E$, the product $\Gamma\times \cO Y$ embeds coarsely into $E\supseteq E \oplus B$, and finally $\cO_\Gamma Y$ piecewise embeds coarsely into $E$ by \cref{global to piecewise}.

For the `if' part recall from Lemma \ref{3.1} that for $y_0\in Y$ belonging to a free orbit, we have $\lim d(ty_0, \gamma ty_0) = |\gamma|$ and in fact the function $t\mapsto d(ty_0, \gamma ty_0)$ is eventually constant. Hence, every finite subset $A$ of $\Gamma$ can be isometrically realised as a subset $Aty_0\subseteq tY$ for sufficiently large $t$. By further increasing $t$ we may assume that $t\geq j({\diam A})$, where $j({\diam A})$ is the appropriate constant from the definition of piecewise property A, piecewise asymptotic dimension, or piecewise coarse embeddability into $E$. Since the considered properties are finitely determined, this ends the proof.
\end{proof}

By combining Theorem \ref{Yamauchi} with Proposition \ref{asdim for the product} and Theorem \ref{SW covering}, one obtains:

\begin{cor}\label{weaker asdim than main asdim} If the action $\Gamma\acts Y$ is Lipschitz and free and $\asdim \cO_\Gamma Y < \infty$, then: 
\[\asdim\Gamma, \asdim \cO Y \leq \asdim (\Gamma \times \cO Y,d_1) = \asdim \cO_\Gamma Y \leq \asdim \Gamma + \asdim \cO Y.\]
\end{cor}
\begin{proof}
Since finite subsets of $\Gamma$ embed isometrically into $\cO_\Gamma Y$ (as just recalled in the proof of \cref{nice group vs cone}), finite determination of the asymptotic dimension (Proposition \ref{determination A asdim}) gives us that $\asdim \Gamma \leq \asdim \cO_\Gamma Y$. By asymptotic faithfulness (\cref{SW covering}), \cref{Yamauchi} gives $\asdim (\Gamma \times \cO Y,d_1) \geq \asdim \cO_\Gamma Y \eqdef k$. 

We will show the opposite inequality $\asdim (\Gamma\times \cO Y,d_1) \leq k$. Fix $R\in \N$. For every $t>0$ we will construct a covering $(W_{t,l})_{l=0}^k$ of $(\Gamma\times tY, d_1)$ with $R$-components uniformly bounded. Let $(U_{t,l})_{l=0}^k$ be a covering of $(tY,d_\Gamma)$ such that diameters of $R$-components of $U_{t,l}$ are bounded by some $S\in \N$ (not depending on $t>0$). By asymptotic faithfulness, there is $t_0>0$ such that for $t > t_0$ the map $\pi$ is isometric, in particular injective, on balls of radius $\max(S,R)$. For such $t$, if we take any $R$-component $C$ of some $U_{t,l}$, its inverse image $\pi^{-1}(C)$ is a disjoint union of isometric copies of $C$. By the above injectivity, two points from different copies must lie at distance at least $R$, and, by contractivity of $\pi$, two points from inverse images of different components $C,C'$ of $U_{t,l}$ are at least $R$-apart as well. Hence, $R$-components of $W_{t,l} \defeq \pi^{-1}(U_{t,l})$ are isometric to $R$-components of $U_{t,l}$, hence $S$-bounded. For $t\leq t_0$, we can consider a covering $(V_l)_l$ of $\Gamma$ with $0\leq l \leq \asdim \Gamma \leq \asdim \cO_\Gamma Y$ such that $R$-components of $V_l$ are $S'$-bounded, and define $W_{t,l} = V_l\times tY$ for $l\leq \asdim \Gamma$ and $W_{t,l} = \emptyset$ otherwise; the diameter of $R$-components of $W_{t,l}$ is clearly bounded by $S'+ t_0 \diam Y$.

So far, the Lipschitzness assumption has not been used. The remaining inequalities were proved in \cref{asdim for the product} utilising this extra assumption.
\end{proof}

Note that using Theorem \ref{BDLM} one can prove that the unified warped cone $\cO^\mathrm{u}_\Gamma Y = \bigcup_{t>0} (tY, d_\Gamma)$ satisfies $\asdim \cO^\mathrm{u}_\Gamma Y \leq \asdim \cO_\Gamma Y + 1$. In fact, a version of \cref{weaker asdim than main asdim} in the `unified' case would also be true, but in the proof (of the inequality $\asdim \Gamma\times \cO^\mathrm{u} Y\leq k$ for $k \defeq \asdim \cO^\mathrm{u}_\Gamma Y$) an argument similar to that in the proof of the finite union theorem \cite{Bell--Dranishnikov} would be needed in order to combine coverings of $\Gamma\times [0,t_0] \times Y$ and $\Gamma\times (t_0,\infty) \times Y$.

We will now proceed to the proof of Theorem \ref{main asdim}, which is a sharper version of the above Corollary \ref{weaker asdim than main asdim} valid for example for free isometric actions on manifolds.

\begin{proof}[Proof of Theorem \ref{main asdim}]
Note that it follows from our assumptions and  the proof of Theorem \ref{qi} that for every $R\in \N$ there is $t_0$ such that for $t>t_0$ any $R$-ball in $(tM, d_\Gamma)$ is quasi-isometric (with constants not depending on $R$) to an $R$-ball in $\Gamma \times \R^m$ (in fact one gets a Lipschitz equivalence and the Lipschitz constant can be taken arbitrarily close to $1$, but we do not need it here nor in Theorem \ref{qi}), and further to an $R$-ball in $\Gamma \times \Z^m$. Consequently, the warped cone has piecewise asymptotic dimension at most $\asdim \Gamma \times \Z^m$. Then, by the theorem of Yamauchi \cite{Yamauchi} (Theorem \ref{Yamauchi} above), we have $\asdim \cO_\Gamma M \leq \asdim \Gamma \times \Z^m$. But this inequality cannot be strict: this would imply that the asymptotic dimension of the family of balls in $\Gamma \times \Z^m$ is strictly smaller than the asymptotic dimension of $\Gamma \times \Z^m$, which is impossible as asymptotic dimension is finitely determined by Lemma \ref{determination A asdim}.

\smallskip
If $(tY,d)$ is an ultrametric space, then for any $R>0$ its $R$-components are $R$-bounded, hence $\asdim \cO Y = 0$ (cf.\ the proof of Theorem \ref{examples from Thiebout and Ana}). In this case, the result follows from \cref{weaker asdim than main asdim}.
\end{proof}

Let us finish this section with the following simple observation.

\begin{prop}\label{equality of topological and asymptotic dimensions} We have $\dimAN Y = \asdimAN\cO Y = \asdim \cO Y$. If the action $\Gamma\acts Y$ is free, then $\dim Y \leq \asdim \cO_\Gamma Y$ and if it is additionally Lipschitz, then the same holds for $\dimAN Y$.
\end{prop}

With a more hands-on approach one can probably drop the freeness assumption.

\begin{proof}[Proof of \cref{equality of topological and asymptotic dimensions}] The estimates $\dimAN Y \geq \asdimAN \cO Y \geq \asdim \cO Y$ follow directly from definitions. For the proof of the opposite inequalities, we can assume $\asdim \cO Y <\infty$ and take $(V_{t,l})_{l=0}^{\asdim \cO Y}$ to be
a covering of $(tY,d)$ with $1$-components of $V_{t,l}$ bounded by some constant $S$. If we identify $tY$ with $Y$, thus scaling the metric by $t^{-1}$, set $V_{t,l}$ becomes a set whose $1/t$-components are $S/t$-bounded, hence establishing $\dimAN Y \leq \asdim \cO Y$, and consequently the equality
$\dimAN Y = \asdimAN\cO Y = \asdim \cO Y$.

Under the freeness assumption, we have $\asdim \cO_\Gamma Y = \asdim (\Gamma \times \cO Y, d_1) \eqdef k$ (whenever the former is finite) by \cref{weaker asdim than main asdim} (this part of the corollary does not use Lipschitzness). Let $(U_{t,l})_{l=0}^{k}$ be an \emph{open} covering of $(\Gamma \times tY, d_1)$ such that $1$-components of $U_{t,l}$ are bounded by some $S\in \N$ and denote the restriction of this covering to $\{e\} \times tY$ by 
$(V_{t,l})_{l=0}^{k}$. By the definition of $d_1$ one always has $d_1\leq d$ (where we identify $\{e\} \times tY$ and $tY$), so $1$-components of $V_{t,l}$ with respect to $d$ are contained in $1$-components taken with respect to $d_1$.

We need the following fact about the metric $d_1$, which is an ingredient of the proof \cite{SW} of \cref{SW covering}: for any $m\in \N$ and $\eps>0$, there exists $t_0=t_0(m, \eps)>0$ such that for $t \geq t_0$ the $d_1$-ball about $(e,ty)$ of radius $m$ is contained in the product of the ball of radius $m$ about $e \in \Gamma$ and the ball of radius $\eps$ about $y\in Y$ with respect to the metric~$d$. Here, we identified $\Gamma\times tY$ with the product of $\Gamma$ and~$Y$. Let us denote the image of $V_{t,l}$ under such identification by $W_{t,l}$. We readily see that for every $\eps>0$ there is $t_0=t_0(S, \eps)$ sufficiently large such that $1/t_0$-components of $W_{t_0,l}$ (with respect to the metric $d$ on $Y$) are $\eps$-bounded, which shows that $\dim Y \leq k$ (as $1/t_0$-components of an open set are open).

If the action is also Lipschitz, then by \cref{cone in product} the infinite cone $\cO Y$ embeds coarsely into $(\Gamma\times \cO Y, d_1)$ and hence \[\dimAN Y = \asdim \cO Y\leq \asdim \Gamma\times \cO Y \leq \asdim \cO_\Gamma Y,\]
where the equality has already been proved, and the latter inequality follows from \cref{weaker asdim than main asdim}.
\end{proof}

\section{\texorpdfstring{Applications to dynamics and C$^*$-dynamics}{Applications to dynamics and C-star dynamics}}\label{section nuclear}

The following was obtained by Jianchao Wu and Joachim Zacharias; we use notation from dimension theory, where $\dim^{+1}(\cdot)$ means $\dim(\cdot)+1$.

\begin{theorem}[Wu--Zacharias, unpublished]\label{WZ} For a free Lipschitz action $\Gamma\acts Y$, we have the following relations between its equivariant asymptotic dimension, denoted $\eqasdim (\Gamma\acts Y)$, and the asymptotic dimension $\asdim \cO_\Gamma Y$ of the warped cone:
\begin{enumerate}
\item\label{bound asdim by eqasdim} $\asdim^{+1} (\cO_\Gamma Y) \leq \asdim^{+1} (\cO Y) \cdot \eqasdim^{+1}(\Gamma\acts Y)$;
\item\label{bound eqasdim by asdim} $\eqasdim(\Gamma\acts Y) \leq \asdim (\cO_\Gamma Y)$.
\end{enumerate}

In particular, if $\asdim \cO Y < \infty$, then 
\[\asdim \cO_\Gamma Y < \infty \iff \eqasdim (\Gamma\acts Y) < \infty.\]
\end{theorem}

In general, equivariant asymptotic dimension is computed relative to a family $\cF$ of subgroups of $\Gamma$ and above (and in the sequel) we consider simply the family $\{\{e\}\}$, the singleton of the trivial subgroup $\{e\}$. We refer the interested reader to \cite{eqasdim} for details, see also \cites{DAD,SWZ}.
We remark that when $\cF$ consists only of the trivial subgroup $\{e\}$ as above, equivariant asymptotic dimension is also called amenability dimension by some authors \cite{SWZ}. Equivariant asymptotic dimension was introduced as a technical condition (without an explicit definition, hence the variety of names in the literature) in the proof of the Farrell--Jones conjecture for hyperbolic groups by Bartels, L\"uck, and Reich \cite{BLR}.

\begin{defi} Let $k\in \N$ and $Y$ be a compact Hausdorff space. A free action $\Gamma \acts Y$ has \emph{equivariant asymptotic dimension} (also known as \emph{amenability dimension}) at most $k$ if for every $R < \infty$ there exists an open covering $\U$ of $\Gamma\times Y$ such that
\begin{enumerate}
\item for every $U\in \U$ and $\gamma\in \Gamma \setminus \{e\}$ we have $\gamma U\cap U = \emptyset$ and $\gamma U\in \U$ (where we use the diagonal $\Gamma$-action on $\Gamma\times Y$);
\item for every $(\gamma,y)\in \Gamma\times Y$ there are at most $k+1$ sets $U\in \U$ containing $(\gamma,y)$;
\item for every $(\gamma,y)\in \Gamma\times Y$ there exists $U\in \U$ containing $B(\gamma, R) \times \{y\}$.
\end{enumerate}
The minimal such $k$ is called the equivariant asymptotic dimension of the action $\Gamma \acts Y$ (if there is no $k$ as above, we say that the dimension is infinite).
\end{defi}

Equivariant asymptotic dimension relates to many other properties like dynamic asymptotic dimension \cite{DAD} and tower dimension \cite{Kerr} of group actions as well as Rokhlin dimension of C$^*$-dynamical systems and nuclear dimension of C$^*$-algebras \cite{SWZ}. For example, the following estimate follows from \cite{DAD}*{Section 8} (cf.\ \cite{SWZ}*{Theorem B} and \cite{Kerr}*{Section 6}).

\begin{thm}[Guentner--Willett--Yu \cite{DAD}] For a free action $\Gamma\acts Y$ on a compact Hausdorff space $Y$ we have:
\[\dim_\textnormal{nuc}^{+1} \left(\operatorname C(Y)\rtimes_\textnormal r \Gamma\right) \leq \eqasdim^{+1}(\Gamma\acts Y) \cdot \dim^{+1} (Y). \]
\end{thm}

Let us come back to the equivariant asymptotic dimension itself. Recall the following result.

\begin{thm}[Szabó--Wu--Zacharias \cite{SWZ}]\label{SWZ} Let $\Gamma\acts Y$ be a free action of a nilpotent group $\Gamma$ on a metrisable space $Y$. Then:
\[\eqasdim^{+1} (\Gamma\acts Y) \leq 3^{\asdim \Gamma}\cdot \dim^{+1} Y,\]
where $\dim Y$ is the Lebesgue covering dimension of $Y$.
\end{thm}

The original result \cite{SWZ}*{Theorem 7.4 with Lemma 8.4} is formulated and proved using the Hirsch length instead of the asymptotic dimension, but they are equal for virtually polycyclic groups (in particular for nilpotent groups) by a classical result of Dranishnikov and Smith \cite{DraSmith}.

The following more general result is a special case of \cite{coarse flow}*{Corollary 1.10}.

\begin{thm}[Bartels \cite{coarse flow}]\label{Bartels}
Let $\Gamma\acts Y$ be a free action of a \emph{virtually} nilpotent group~$\Gamma$ on a metrisable space $Y$ of finite covering dimension. Then:
\[\eqasdim (\Gamma\acts Y) < \infty.\]
\end{thm}

Guentner, Willett, and Yu introduced another notion of dimension, called dynamic asymptotic dimension and denoted $\operatorname{dad}(\Gamma \acts Y)$, which is bounded above by the equivariant asymptotic dimension \cite{DAD}. Coverings $\U = \U(R)$ in its definition are coverings of $Y$ and it is required that for every $U\in \U$ the equivalence classes induced by identifying two points $y,y'\in U$  such that $y'=\gamma y$ for some $\gamma\in \Gamma$ with $|\gamma|\leq R$ have uniformly bounded cardinalities; see \cite{DAD} for details.

Willett was interested in improving the above upper bounds, which is a problem he formulated in particular during a seminar lecture in Warsaw in June 2017. One always has $\asdim \Gamma \leq \eqasdim (\Gamma \acts Y)$ \cite{DAD} and $\dim Y$ should probably also contribute to $\eqasdim (\Gamma \acts Y)$ (compare \cref{possibilities for dad}), so bounds involving $\asdim \Gamma$ and $\dim Y$ as in \cref{SWZ} are expected.

\begin{problem}[Willett] Improve these upper bounds on the equivariant asymptotic dimension.
\end{problem}

The next theorem, which is the main result of this section, addresses the problem, giving a common bound in the nilpotent and virtually nilpotent case. For a zero-dimensional space $Y$ and \emph{any} group, this bound becomes an equality.

\begin{thm}\label{corollary for eqasdim} Let $Y$ be a compact metrisable space of finite covering dimension. Let an action $\Gamma \acts Y$ be free and have finite equivariant asymptotic dimension (e.g.\ any free action of a virtually nilpotent group). Then:
\[\asdim \Gamma \leq \operatorname{dad}(\Gamma \acts Y) \leq \eqasdim (\Gamma\acts Y) \leq \asdim \Gamma + \dim Y.\]
\end{thm}

For 0-dimensional spaces in a very recent work \cite{Kerr} Kerr obtained the following equalities:
\[ \dim_\textnormal{tow}(\Gamma\acts Y) = \dim_\textnormal{ftow}(\Gamma\acts Y) = \operatorname{dad}(\Gamma\acts Y) = \eqasdim(\Gamma\acts Y),\]
where the first two notions of dimension, tower dimension and fine tower dimension, are introduced in \cite{Kerr}.
We conclude that in fact all these can acquire only two values: $\asdim \Gamma$ or infinity.

\begin{cor}\label{corollary from Kerr} Assume that $Y$ is a compact metrisable space of covering dimension zero and $\Gamma\acts Y$ is free. We have:
\[ \dim_\textnormal{tow}(\Gamma\acts Y) \!=\! \dim_\textnormal{ftow}(\Gamma\acts Y) \!=\! \operatorname{dad}(\Gamma\acts Y) \!=\! \eqasdim(\Gamma\acts Y) \in \{\asdim \Gamma, \infty\}.\]
\end{cor}

This sets up, in the zero-dimensional case, the following question of Willett, expressed in particular during his lectures in Southampton in March 2017 \cite{metric2011-southampton}.

\begin{question}[Willett]\label{possibilities for dad} Is it true that for any action $\Gamma \acts Y$ one has \[\operatorname{dad}(\Gamma\acts Y)\in \{\asdim \Gamma, \infty\}?\]
\end{question}

Note that the answer is also positive for actions of $\Z$, at least if $Y$ is infinite and $\Z\acts Y$ is minimal, in which case $\operatorname{dad}(\Z\acts Y) = 1$ by {Theorem 3.1} of \cite{DAD}, inspired by ideas of I.\ Putnam.
Let us additionally remark that by {Theorem 6.6} of \cite{DAD} every group~$\Gamma$ of finite asymptotic dimension admits a free action on a Cantor set with all the notions of dimension from \cref{corollary from Kerr} finite, namely, equal to $\asdim \Gamma$.

In order to prove \cref{corollary for eqasdim} we will need the following lemma.

\begin{lem}\label{change of metric} Let $\Gamma\acts Y$ be an action of a finitely generated group on a compact metrisable space $Y$. There exists a metric $d$ on $Y$ such that $\Gamma$ acts on $(Y,d)$ by Lipschitz homeomorphisms. Moreover, it can be arranged that $\dimAN (Y,d) = \dim Y$.
\end{lem}
\begin{proof} By results of Nagata \cites{Nagata} and Assouad \cite{Assouad} a compact metrisable space $Y$ has topological dimension at most $k$ if and only if it admits a metric $d_0$ such that $(Y,d_0)$ has Assouad--Nagata dimension at most $k$, hence we can take $d_0$ such that $\dimAN(Y,d_0)=\dim Y$.

Let $c_0$ be the maximum value acquired by $d_0$. We will construct a function $c\colon [0,c_0]\to [0,1]$ such that $d = c \circ d_0$ is the required metric. Let us inductively define an auxiliary sequence $c_n$ by the formula
\[c_{n+1} = \min \left\{c_n/3 ,\ \min\left\{ d_0(sy, sy') \st s\in S,\ d_0(y,y')\geq c_n\right\}\right\},\]
where the inner minimum is well defined and positive by compactness of $Y$. Define $c(c_n) = 2^{-n}$, $c(0)=0$ and extend it in an affine way on intervals $[c_{n+1}, c_n]$. Note that we have the following inequality of difference quotients:
\[0 < \frac{2^{-n} - 2^{-n-1}}{c_n-c_{n+1}} \leq \frac{2^{-n-1}}{2c_n/3} = \frac{2^{-n-2}}{c_{n}/3} \leq \frac{2^{-n-2}}{c_{n+1}} \leq \frac{2^{-n-1}-2^{-n-2}}{c_{n+1}-c_{n+2}},\]
and so $c$ is a concave and increasing function. Thus, the composition $d = c \circ d_0$ is still a metric, that is, it satisfies the triangle inequality (see e.g.\ \cite{coarse triviality}).

Observe that the action of every generator $s\in S$ is $4$-Lipschitz with respect to~$d$. Indeed, let $y\neq y' \in Y$, $s\in S$ and let $n$ be such that $d_0(sy, sy') \in [c_{n+1}, c_n]$. Then, because $s^{-1}\in S$, we have 
\[d_0(y,y') = d_0(s^{-1}sy, s^{-1}sy') \geq c_{n+2}\] and hence \[d(sy,sy') \leq 2^{-n} = 4 \cdot 2^{-n-2} \leq 4 d(y,y').\]

Our change of metric does not increase the Assouad--Nagata dimension (cf.\ \cite{coarse triviality} for similar reasonings for asymptotic Assouad--Nagata dimension). Indeed, it will follow from the observation that $R$-components of subsets of $Y$ with respect to~$d$ are the same as $c^{-1}(R)$-components with respect to~$d_0$, so fix $R\in(0,1]$. For the metric $d_0$ there exists $K<\infty$ not depending on $R$ such that there exist a covering $U_0, \cdots, U_{\dim Y}$ of $Y$ such that $c^{-1}(R)$-components of $U_i$ are $K\cdot c^{-1}(R)$-bounded. Since $c_{j+1}/c_j \leq 1/3$, if $n$ is such that $c^{-1}(R) \in [c_{n+1},c_n]$, then $K\cdot c^{-1}(R) \leq c_{n-\ceil{\log_3 K}}$ (formally, one needs to extend the sequence $(c_j)$ by putting $c_j = c_0 \cdot 3^{-j}$ for $j<0$; we also need to extend the map $c$, for example by putting $c(r)=1$ for $r\geq c_0$). Coming back to the metric $d$, we have $R = c(c^{-1}(R))\geq 2^{-n-1}$ and 
\[c(K\cdot c^{-1}(R)) \leq c(c_{n-\ceil{\log_3 K}}) \leq 2^{-n+\ceil{\log_3 K}},\] so $R$-components of $U_i$ are $(2^{\ceil{\log_3 K}+1}\cdot R)$-bounded.
\end{proof}

\begin{proof}[Proof of \cref{corollary for eqasdim}] 
Consider a metric $d$ from \cref{change of metric} such that $\dimAN (Y,d)$ equals $\dim Y$. 
If $\Gamma$ is virtually nilpotent, then indeed by \cref{Bartels} and the assumption that $\dim Y<\infty$, we know that $\eqasdim (\Gamma\acts Y)$ is finite. By \cref{equality of topological and asymptotic dimensions} we also have $\asdim \cO (Y,d) = \dimAN (Y,d) = \dim Y < \infty$. Hence, by Theorem \ref{WZ} item \eqref{bound asdim by eqasdim}, $\asdim(\cO_\Gamma Y)$ is finite as well. By Corollary \ref{weaker asdim than main asdim} it is at most $\asdim \Gamma + \asdim \cO Y = \asdim \Gamma + \dim Y$. By Theorem \ref{WZ} item \eqref{bound eqasdim by asdim}, we conclude that $\eqasdim (\Gamma\acts Y) \leq \asdim \Gamma + \dim Y$.

The remaining inequalities follow from \cite{DAD}. More precisely, we have:
\[\asdim \Gamma = \operatorname{dad}(\Gamma \acts \beta\Gamma) \leq \operatorname{dad}(\Gamma \acts Y) \leq \eqasdim (\Gamma\acts Y),\]
where the equality is a part of Theorem 6.5 of \cite{DAD}, the former inequality follows (see the proof of the same theorem) from the universal property of the Stone--\v{C}ech compactification $\beta \Gamma$ of $\Gamma$, and the latter inequality follows from Theorem 4.11 of the same work.
\end{proof}

\section*{Acknowledgements} The author is grateful to Jianchao Wu and Joachim Zacharias for permission to use their unpublished result, \cref{WZ}. The author thanks Kang Li for discussions on different notions of dimension and suggesting improvements in the exposition of \cref{section nuclear}. The author is indebted to Damian Orlef for inspiring discussions on diophantine approximations, in particular for suggesting a proof of \cref{higher tori}. Both authors of the appendix thank Piotr Nowak, whose question had initiated the collaboration that resulted in the appendix. They are grateful to the referee for carefully reading the manuscript and offering many helpful suggestions.

The author was partially supported by Narodowe Centrum Nauki grant Preludium number 2015/19/N/ST1/03606.

\appendix

\section{Approximating the badly approximable\\
by Dawid Kielak and Damian Sawicki}\label{appendix}

In the appendix we prove the following results.

\begin{thm}\label{main appendix} Let $\alpha \in (0,1)$ and let $\Z$ act on the circle $\T^1 = \R/\Z$ by shifts by~$\alpha$.
Then the warped cone $(t\T^1, d_\Z)_{t>0}$ is quasi-isometric to the collection of tori $(\T^2, \tau d)_{\tau > 0}$ if and only if $\alpha$ is a restricted irrational number.
\end{thm}

By a restricted irrational number we mean an irrational number with a bounded sequence of integers as its continued fraction representation, see \cref{restricted definition} for details.

We also offer a result about actions of the infinite dihedral group
\[D_\infty = \Z/2\Z * \Z/2\Z = \{1,r\} * \{1,r'\}.\]

\begin{cor}\label{dihedral spheres} Let $D_\infty\acts \T^1$ be such that $r$ and $r'$ act by non-trivial reflections.
Then the warped cone $(t\T^1, d_{D_\infty})_{t>0}$ is quasi-isometric to the collection of spheres $(\SS^2, \tau d)_{\tau > 0}$ if and only if $r'r$ acts by a rotation by a restricted irrational number.
\end{cor}

In fact, Theorem \ref{main appendix} generalises to actions of $\Z^m$. For $\alpha = (\alpha_1, \ldots, \alpha_m)\in (0,1)^m$ we consider the action $\Z^m\acts_\alpha \T^1$ given by $(n_1, \ldots, n_m).z = z + \sum n_i \alpha_i$.

\begin{prop}\label{higher tori} Let $m\in \N_+$. There exist strictly monotonic sequences $(t_i)$ and $(\tau_i)$ in $\N$ and uncountably many $m$-tuples $(\alpha_1,\ldots, \alpha_m)$ such that the subcollection $(t_i\T^1, d_{\Z^m})_i$ of the warped cone over $\Z^m\acts_\alpha \T^1$ is quasi-isometric to the collection of tori $(\T^{m+1}, \tau_i d)_i$.
\end{prop}

For $\alpha = (\alpha_1, \ldots, \alpha_m)\in (0,1)^m$ we define the action $\Z^m\acts_\alpha \T^m$ by the formula $(n_i)_i. (z_i)_i = (z_i + n_i \alpha_i)_i$. When $m=1$, we get the same action as in Theorem~\ref{main appendix}, so there is no ambiguity.

\begin{ex} There exist strictly monotonic sequences $(t_i), (\tau_i)$, and $(r_i)$ in $\N$ and uncountably many triples $\beta = (\beta_1,\beta_2, \beta_3)\in (0,1)^3$ such that for any pair of restricted irrationals $\alpha = (\alpha_1, \alpha_2)$ and actions $\Z^3\acts_\beta \T^1$ and $\Z^2\acts_\alpha \T^2$ we get quasi-isometries
\[(t_i\T^1, d_{\Z^3})_i \simeq (\tau_i\T^2, d_{\Z^2})_i \simeq (\T^4, r_i d)_i.\]
\end{ex}
The reason for the above is that $(\tau_i\T^2, d_{\Z^2})$ can be decomposed as a Cartesian product $(\tau_i \T,d_\Z) \times (\tau_i \T,d_\Z)$, where the $\Z$-action on each factor is given by rotations by $\alpha_1$ or $\alpha_2$ and each factor $(\tau_i \T,d_\Z)$ is quasi-isometric to $(\T^2,r_i d)$ by \cref{main appendix}.

All of the results above show that there exist quasi-isometric warped cones (or subcollections thereof) over actions of non-quasi-isometric groups on non-homeo\-morphic spaces (actions of infinite groups are also not orbit equivalent in any sense to actions of the trivial group). In particular, the $\Z^m$ and $\Z^n$ factors are necessary in Theorem~\ref{qi}.

The above results rely on the following, more technical, result.

\begin{lem}\label{technical appendix} Let $m\in \N_+$, $(\alpha_1, \ldots, \alpha_m)\in (0,1)^m$ and $0<K<\infty$. Consider a pair $(q, l) \in \N_+\times [1,\infty)$ with $q\leq l$ and assume that there exist $(p_i)\in \{1,\ldots, q-1\}^m$ such that 
\begin{itemize}
\item $|q\alpha_i - p_i| < K/l$ for all $i\in \{1,\ldots, m\}$, and
\item the group $\Z/q\Z$ admits a decomposition $\Z/q\Z  \cong \bigoplus_i \Z/l_i\Z$ such that the image of $p_i$ in $\Z/q\Z$ under this isomorphism is a generator of $\Z/l_i\Z$.
\end{itemize}
Then, for the action $\Z^m\acts_\alpha \T^1$, the level $(lq\T^1, d_{\Z^m})$ of the warped cone is quasi-isometric to $(\T^1,ld) \times \prod_i (\T^1, l_i d)$, with quasi-isometry constants depending only on the dimension $m$ and the quality of approximation $K$.
\end{lem}
\begin{proof}
Let us put $t=lq$ and denote the $\ell_1$-metric on $\T^{m+1}=(\T^1,ld) \times \prod_i (\T^1, l_i d)$ by $D$.
In order to construct a quasi-isometry $(t\T^1, d_{\Z^m}) \to (\T^{m+1},D)$, we will use an intermediate space $(t\T^1,d_{\Z^m}')$, which is defined as $t\T^1$ with the metric warped with respect to the action $\Z^m\acts_\beta \T^1$ by the $m$-tuple of angles $\beta = (\beta_1, \dots, \beta_m)$ with $\beta_i = p_i/q$. Since we will mainly work with the $\beta$-action, we reserve the notation $(n_i).z$ for the image of $z\in \T^1$ under the $\beta$-action of $(n_i)\in \Z^m$, denoting the image under the respective $\alpha$-action by $\alpha_{(n_i)}(z)$.

First note that the identity map $\id\colon (t\T^1, d_{\Z^m}) \to (t\T^1, d_{\Z^m}')$ is bi-Lipschitz. Indeed, recall that for isometric actions $\Gamma\acts Y$, we have 
\[ d_\Gamma(ty, ty') = \min_{\gamma \in \Gamma} \big(|\gamma| + td(\gamma y, y') \big) \]
by Lemma \ref{2.1}, and so we get the following inequality (we suppress $t$ in the notation for points in $t\T^1$ as $t=lq$ is fixed):
\begin{align*}
d_{\Z^m}(z,z') & = \min_{(n_i)\in \Z^m} \big( |(n_i)| + t d(\alpha_{(n_i)}(z), z' )\big)\\
& = \min_{(n_i)\in \Z^m} \big(\sum |n_i| + t d(z + \sum n_i\alpha_i, z'  ) \big)\\
&  \leq \min_{(n_i)\in \Z^m} \big(\sum |n_i| + t d(z + \sum n_i p_i/q, z') + t  \sum |n_i| \cdot |\alpha_i - p_i/q| \big)\\
& \leq    \min_{(n_i)\in \Z^m} \big(\sum |n_i| + t d(z + \sum n_i p_i/q, z') + K \sum |n_i| \big)\\
& \leq (K + 1) d_{\Z^m}'(z,z')
\end{align*}
(the converse estimate has the same proof).

Observe that, by the Chinese Remainder Theorem, $p_i = \prod_{j\neq i} l_j \cdot p_i'$ for some $p_i' \in \{1,\ldots, l_i - 1\}$ which generates $\Z/{l_i}\Z$. Let  $r_i'\in \{1,\ldots, l_i - 1\}$ be such that $r_i = r_i'\cdot \prod_{j\neq i} l_j$
is an inverse of $p_i'$ modulo $l_i$.
Our quasi-isometry $\iota \colon (t\T^1, d_{\Z^m}') \to (\T^1,ld) \times \prod_i (\T^1, l_i d)$ is given by
\[\iota(z) = (q z,  r_1 z, \ldots, r_m z). \]

We will first check that some metric neighbourhood (depending only on $m$) of the image of $\iota$ is the whole of $(\T^{m+1},D)$. Note that $q \beta_i \equiv 0 \pmod 1\text{ and also } r_i \beta_j \equiv 0 \pmod 1$
for $i\neq j$, because $\beta_j = p_j'/l_j$  and $l_j$ divides $r_i$. Consequently, we have the following formula
\begin{equation}\label{iota explained}
\iota((n_i).z) = (qz, r_1 z+n_1/l_1, \ldots, r_m z + n_m/l_m).
\end{equation}
Thus, for any $\zeta\in (\T^1,ld) \times \prod_i (\T^1, l_i d)$, we can find $z\in\T^1$ and $(n_i)$ such that $\iota((n_i).z)$ has the first coordinate equal to the first coordinate of $\zeta$ and the remaining coordinates of $\iota((n_i).z)$ and $\zeta$ differ by distances at most $1/2$.

It remains to check that $\iota$ is a bi-Lipschitz map onto its image. Let $z,z'\in (t\T^1, d_{\Z^m}')$ and let $(n_i)\in \Z^m$ be such that $d_{\Z^m}'(z,z') =  |(n_i)| + td((n_i).z,z')$. Then by the triangle inequality we have
\[D\big(\iota(z), \iota(z')\big) \leq D(\iota(z), \iota((n_i).z)) + D(\iota((n_i).z), \iota(z'))\]
and, using formula \eqref{iota explained}, the first summand can immediately be bounded above: $D(\iota(z),\iota((n_i).z)) \leq |(n_i)|.$
As for the second summand, we obtain
\begin{align*}
D(\iota((n_i).z), \iota(z')) &= ld(q \cdot (n_i). z, q \cdot z') + \sum_j l_j d(r_j \cdot (n_i). z, r_j \cdot z') \notag\\
& \leq \left(lq + \sum l_j r_j\right) d((n_i). z, z') \leq (m+1) t d((n_i).z, z'),
\end{align*}
as $r_j, l_j\leq q \leq l$ and $t=lq$.
We have just obtained that $\iota$ is Lipschitz:
\[ D(\iota(z), \iota(z')) \leq |(n_i)| + (m+1) td((n_i).z,z') \leq (m+1) d_{\Z^m}'(z,z').\]

Note that $\iota$ is a group homomorphism and all metrics are translation invariant. Hence, when proving the converse estimate, we can assume that one of $z,z'$ is zero. For $z\in \T^1$ let $\theta \in (-1/(2q), 1/(2q)]$ and $n_i\in (-l_i/2, l_i/2]\cap \Z$ be the unique numbers such that
\[z \equiv \sum n_i\beta_i + \theta \pmod 1.\]
We have
\begin{align*}
d_{\Z^m}'(0,z) &= \min_{(m_i)\in \Z^m} \big(|(m_i)| + td(\sum m_i\beta_i, z)  \big) \\
&\leq |(n_i)| + td(\sum n_i\beta_i, z)  =   |(n_i)| + t|\theta| ,
\end{align*}
and also
\[D(\iota(0), \iota(z)) = ld(0,q\theta) + \sum_i l_i d(0, n_i/l_i + r_i\theta).\]
If $2 t|\theta| \geq \max_i |n_i|$, then $(2m+1)t\theta\geq |(n_i)|+t\theta \geq d_{\Z^m}'(0,z)$ and we can bound:
\begin{align*}
D(\iota(0), \iota(z)) \geq ld(0, q\theta) =  lq |\theta| \geq d_{\Z^m}'(0,z)/(2m+1)
\end{align*}
(where the equality follows from the assumption that $\theta \in (-1/(2q), 1/(2q)]$);
otherwise, for $j$ such that $|n_j| = \max_i |n_i|$, we have:
\begin{align*}
D(\iota(0), \iota(z)) & \geq l_j d(0, r_j z) = l_j d(0, n_j/l_j + r_j\theta) \nonumber \\
&\geq l_j d(0, n_j/l_j) - l_j d(0, r_j\theta)= |n_j| - l_j r_j |\theta|\\
&\geq |n_j| - t|\theta| > |n_j|/2 > d_{\Z^m}'(0,z)/(2m+1)
\end{align*}
(where the second equality again uses assumptions on $\theta$ and on $n_i$).
We have just proved
\[d_{\Z^m}'(z,z')/(2m+1) \leq D(\iota(z), \iota(z')) \leq (2m+1) d_{\Z^m}'(z,z').\qedhere\]
\end{proof}

We will now deduce results \ref{main appendix}--\ref{higher tori} from Lemma \ref{technical appendix}. For that, we need a bit of number theory.

Given a sequence of integers $(a_0, a_1, \ldots)$ with $a_i > 0$ for $i > 0$ one can form a continued fraction
\[
[a_0,a_1,\ldots] = a_0 + \frac{1}{a_1 + \frac{1}{a_2+\frac{1}{a_3 + \frac{1}{\cdots}}}}.
\]
The value of a continued fraction is the limit of the sequence of \emph{convergents}:
\[a_0,\ a_0 + \frac{1}{a_1},\ a_0 + \frac{1}{a_1 + \frac{1}{a_2}},\
a_0 + \frac{1}{a_1 + \frac{1}{a_2+\frac{1}{a_3}}},\ \ldots\]
This establishes a bijection between irrational numbers and continued fractions. Let us denote by $p_i\in \Z$ and $q_i\in \N_+$ the integers such that $p_i/q_i$ is the $i$th convergent of the given continued fraction
and $q_i$ is minimal. In particular $p_0=a_0$ and $q_0=1$.

\begin{defi}\label{restricted definition} An irrational number $\alpha$ is called \emph{restricted} if the sequence $(a_i)$ is bounded for the representation $\alpha = [a_0,a_1,\ldots]$.
\end{defi}

We remark that such irrational numbers are precisely those that are badly approximable, that is, there exists $c>0$ such that for all positive integers $p,q$ we have $|\alpha - p/q| > c/q^2$. A concrete example is any algebraic number of degree~$2$. The term `badly approximable' comes from the fact that such a quadratic approximation is considered a bad approximation: by the Hurwitz theorem, every irrational $\alpha$ admits infinitely many rationals $p/q$ with $|\alpha - p/q| < 1/q^2$.

Indeed, recall that we have the following formulae (see, for instance, \cite{Sierpinski}*{Chapter~6}):
\begin{align*}
p_i &= a_i p_{i-1} + p_{i-2}\\
\numberthis\label{recurrence q}q_i &= a_i q_{i-1} + q_{i-2}\\
\frac{p_{i+1}}{q_{i+1}} - \frac{p_i}{q_i} &= \frac{(-1)^i}{q_i q_{i+1}}
\end{align*}
Hence, both sequences $(p_i)$ and $(q_i)$ are increasing (for $\alpha>0$). We also have $p_i/q_i<\alpha$ for even $i$ and $p_i/q_i > \alpha$ for odd $i$, so we can bound:
\begin{equation}\label{approximation by convergents}
\left|\alpha - \frac{p_i}{q_i}\right| < \left|\frac{p_{i+1}}{q_{i+1}} - \frac{p_i}{q_i}\right| = \frac{1}{q_i q_{i+1}}.
\end{equation}

\begin{proof}[Proof of Theorem \ref{main appendix}]
Let $\alpha = [a_0, a_1, \ldots] \in (0,1)$ be an irrational number. Since $q_0=1$, for every $l_0\in [1,\infty)$ there exists $i$ such that $q_i \leq l_0$ and $q_{i+1} > l_0$. Then, by \eqref{approximation by convergents} we have $| q_i \alpha - p_i| \leq 1/q_{i+1} \leq 1/l_0.$

By \eqref{recurrence q} we have
\[q_i = a_i q_{i-1} + q_{i-2} \leq (a_i+1) q_{i-1},\]
so, when $A=\sup_i a_i$ is finite, we see that $q_i \geq l_0/(A+1)$. Consequently, for every $t\geq 1$ there exists $l_0=\sqrt t$ and $i$ such that $t/(A+1) \leq l_0 q_i \leq t$ and $|q_i \alpha - p_i|\leq 1/l_0$. If $w=t/(l_0 q_i)$, then, for $l = w l_0$, we have $t=lq_i$ and \[|q_i \alpha - p_i|\leq w/(wl_0) = w / l \leq (A+1)/l.\]
That is, for every $t\geq 1$ there exist $l,p,q$ satisfying $q\leq l$, $|q\alpha - p|\leq (A+1)/l$ and $lq=t$. Theorem \ref{main appendix} follows from Lemma \ref{technical appendix} (for $m=1$), because the torus $(\T^2, \sqrt{t}d)$ is (uniformly) Lipschitz equivalent to the torus $(\T^1, ld) \times (\T^1, qd)$, when $l= w \sqrt t$ and $q=\sqrt t/ w$ for $w\in[1,A+1]$.

\smallskip
Note that when $\alpha$ is rational, then the warped cone is quasi-isometric to a collection of circles (by Example \ref{examples to non-rigidity} item \eqref{divide by finite} or \cite{Kim}), which is clearly not quasi-isometric to a collection of tori (for instance, they have different asymptotic dimension).

\smallskip
When $\alpha$ is irrational, but not restricted, that is, $A=\infty$, then $q_{i+1}/q_i$ is unbounded as a sequence in $i$. Consequently, when we put $m=1$, $l=q_{i+1}$ and $q=q_i$ in Lemma \ref{technical appendix}, then the respective levels are quasi-isometric to the more and more `unbalanced' tori $(\T^2, q_{i+1}d\times q_i d)$. They cannot be $(C,B)$-quasi-isometric to ordinary tori $(\T^2, t_i d)$.

Indeed, quasi-isometry approximately preserves diameters, that is, any torus $(\T^2, t_i d)$ quasi-isometric to $(\T^2, q_{i+1}d\times q_i d)$ must satisfy (say that we are using the $\ell_\infty$-metrics):
\[C^{-1} t_i - 2B \leq q_{i+1} \leq Ct_i + 2B\]
for some $C\geq 1$, $B\geq 0$ not depending on $i$. For $N>0$, consider the following invariant $v_N$ of metric spaces: $v_N(X)$ is the maximal cardinality of an $N$-separated subset of $X$. Since the Riemannian volume of $(\T^2, t_i d)$ is $t_i^2$ and for $(\T^2, q_{i+1}d\times q_i d)$ it equals $q_{i+1}q_i$, we also have $v_N(\T^2, t_i d)\asymp t_i^2/N^2$ and $v_N(\T^2, q_{i+1}d\times q_i d) \asymp q_{i+1}q_i/N^2$ for $N\leq t_i,q_i$, where $\asymp$ denotes equality up to some universal multiplicative constants. If there exists a $(C,B)$-quasi-isometric embedding \[(\T^2, t_i d)\to (\T^2, q_{i+1}d\times q_i d),\] then $v_N(\T^2, t_i d) \leq v_{N/C-B}(\T^2, q_{i+1}d\times q_i d)$, which contradicts the fact that \[\liminf_i( q_{i+1} q_i/t_i^2 ) = 0. \qedhere\]
\end{proof}

Kim proved that if two continued fractions $\alpha=[a_i], \beta=[b_i]$ have the same `tail', that is, there exists $k,l\in \N$ such that $a_{k+i} = b_{l+i}$ for all $i\in \N$, then the respective warped cones are quasi-isometric \cite{Kim}. This tail equivalence relation gives a classification of rotation C$^*$-algebras $\operatorname C(\T^1)\rtimes_\alpha \Z$ up to Morita equivalence. Kim asked whether the tail relation also provides the quasi-isometric classification of warped cones (over actions by rotations). In particular, the affirmative answer to her question would imply that each set of angles $\alpha$ giving a quasi-isometric warped cone is countable;
note that this is indeed the case for rational angles $\alpha$, since they form such a set \cite{Kim}. Theorem \ref{main appendix} identifies another set of this kind, which  can be thought of as another step towards the classification (second step out of infinitely many...). Our result shows that countability does not hold in general, and hence gives a negative answer to Kim's question.

\begin{cor} There exist uncountably many angles $\alpha \in (0,1)$  yielding the same quasi-isometry type of a warped cone over the rotation by $\alpha$. In particular, the equivalence relation on $(0,1)$ given by quasi-isometry of warped cones (over the respective rotations) is strictly coarser than the tail equivalence relation for continued fractions.
\end{cor}

The following remark shows that all warped cones over irrational rotations have something in common, even though, by a result of Kim, there are infinitely many quasi-isometry types among them.

\begin{rem}
When $\alpha\in \R\setminus \Q$ is not necessarily restricted, we still have a quasi-isometry on a subsequence $(q_i^2\T^1,d_\Z)_{i\in \N} \simeq (\T^2,{q_i}d)_{i\in \N}$.
Indeed, by~\eqref{approximation by convergents} it suffices to apply \cref{technical appendix} to numbers $m=1$, $l=q_i$ and $q=q_i$.
\end{rem}

\begin{proof}[Proof of Corollary \ref{dihedral spheres}]
Fix $t>0$. For any metric space $(Y,d)$ with an action of a group generated by two elements $r,\eps$ we have the equality of warped metrics
\[d_{\langle \eps,r\rangle} = (d_{\langle \eps\rangle})_{\langle r\rangle},\]
where the respective actions of $\langle \eps\rangle$ and $\langle r\rangle$ are given by restrictions of the action of $\langle \eps,r\rangle$, and both metrics are considered on $tY$ for a fixed $t>0$.

Consider the following generators of $D_\infty$: $r$ and $\eps=r'r$.
It follows that $d_{D_\infty} = (d_{\Z})_{\langle r\rangle}$, where $d_\Z$ is the metric on $t\T^1$ as in Theorem \ref{main appendix} and $\Z=\langle\eps\rangle$. Let us assume for simplicity that $r$ acts by complex conjugation if we consider $\T^1$ as a subset of $\C$. 

Since we have a bi-Lipschitz equivalence of $d_\Z$ and $d_\Z'$ (as in the proof of Lemma \ref{technical appendix}), we get a bi-Lipschitz equivalence $(d_\Z)_{\langle r\rangle} \simeq (d_\Z')_{\langle r\rangle}$.

Consider the action of $\langle r\rangle$ on $\T^2$ by complex conjugation on both coordinates (via the inclusion $\T^2\subseteq \C^2$). Clearly, $\iota\colon t\T^1\to \T^2$ from the proof of Lemma \ref{technical appendix} is equivariant with respect to the actions of $\langle r\rangle$. Since it is also a bi-Lipschitz embedding with respect to the metric $d_\Z'$ on $t\T^1$ and $\tau d$ on $\T^2$ for $\tau=\sqrt t$, we get a bi-Lipschitz equivalence $(t\T^1,(d_\Z')_{\langle r\rangle}) \simeq (\tau\im(\iota), d_{\langle r\rangle})$. However, since $B$-neighbourhood of $\im(\iota)$ is the whole of $\tau \T^2$ for $B=(A+1)/2$ (we proved that for the metric $D$ from the proof of Lemma \ref{technical appendix} and $B=1/2$; the metric $D$ is $(A+1)$-Lipschitz equivalent to the metric $\tau d$, so our claim holds for $B=(A+1)/2$ and the metric $\tau d$, and consequently for the warped metric $d_{\langle r\rangle}$ on $\tau\T^2$), the inclusion $(\tau\im(\iota), d_{\langle r\rangle}) \subseteq (\tau \T^2,  d_{\langle r\rangle})$ is a quasi-isometry. (Note that, by Lemma \ref{2.1}, the warped metric on $\im(\iota)$ and the restriction to $\im(\iota)$ of the warped metric on $\T^2$ are equal, because the action is isometric.)

The last point is to observe that $(\tau \T^2,  d_{\langle r\rangle})$ is quasi-isometric to $(\tau \SS^2, \delta)$ (where $\delta$ is, say, the standard geodesic metric on $\SS^2$). This follows from Example \ref{examples to non-rigidity} item \eqref{divide by finite}, and the fact that the quotient of $\T^2$ by the above action is bi-Lipschitz equivalent to $\SS^2$.

\smallskip
For the `only if' part, observe that if $\eps$ acts by a rational rotation, then the action factors through a finite group and the respective quotient (see Example \ref{examples to non-rigidity} item \eqref{divide by finite}) is an interval; however no unbounded family of intervals is quasi-isometric to a family of spheres. 
When $\eps$ acts by an irrational rotation, one can reason similarly as in the proof of Theorem \ref{main appendix}, because the quotient of an (unbalanced) torus $(\T^2, ad\times bd)$ by the action of $\langle r \rangle$ has diameter equal to the diameter of this torus, that is, $\max(a,b)/2$, and the value of the invariant $v_N$ (from the proof of Theorem \ref{main appendix}) for the quotient is bounded above by its value for the torus, which is approximately $ab/N^2$. However, the $v_N$ invariant of the sphere of diameter $c$ is roughly $c^2/N^2$.
\end{proof}

Note that the actions in Corollary \ref{dihedral spheres}, where both generators are required to act by reflections, are the only interesting isometric actions $D_\infty\acts\T^1$. If one of the generators acts by the central symmetry or the identity, then the action factors through the finite group $\Z/2\Z \times \Z/2\Z$.

Finally, let us prove Proposition \ref{higher tori}.

\begin{proof}[Proof of Proposition \ref{higher tori}]
Fix $m\in \{2,3,4, \ldots\}$. Choose $b_1\leq \ldots \leq  b_m$ odd primes. For $i\in \{1,\ldots, m\}$ let 
\[\alpha_i = \sum_{n=1}^\infty \frac{N_{i,n}}{D_{i,n}} \text{ for } D_{i,n}=\exp\Big(\log(b_i){\floor{m^{2n} \log_{b_i}(2)}}\Big)\]
where $N_{i,n}\in \{1, \ldots, b_i - 1\}$ (since $b_i$ is greater than $2$, there are uncountably many such real numbers $\alpha_i$). In other words, the denominators $D_{i,n}$ are the largest powers of $b_i$ bounded by $2^{m^{2n}}$.

Fix $k\in \N_+$. The partial sum $\beta_{i,k} = \sum_{n=1}^k \frac{N_{i,n}}{D_{i,n}}$
provides an approximation of $\alpha_i$ with error bounded by $2b_i^2/2^{m^{2k+2}}$. Indeed, if $D_{i,n}$ were defined as $2^{m^{2n}}$, then
using the bound $2^{m^{2n}} \geq 2^{m^{2k+2} + (n-k-1)}$ for every $n \geq k+1$
and the formula for the sum of a geometric series we would bound the error by $2b_i/2^{m^{2k+2}}$, and by definition we have $D_{i,n}\geq 2^{m^{2n}} / b_i$.
For every $i\in \{1,\ldots, m\}$ we express $\beta_{i,k}$ as a fraction with the common denominator $q_k = \prod_j D_{j,k}$ and numerator $p_{i,k}$. Since $D_{i,k}\leq 2^{m^{2k}}$, we have
\begin{align*}
|q_k \alpha_i - p_{i,k} | &\leq q_k \cdot 2b_i^2/2^{m^{2k+2}} \leq \left(2^{m^{2k}}\right)^m \cdot 2b_i^2/2^{m^{2k+2}} \\
&= \frac{2b_i^2}{(2^{m^{2k+1}})^{m-1}}
\leq \frac{2b_m^2 }{2^{m^{2k}}},
\end{align*}
that is, for $l=2^{m^{2k}}$, $q=q_k$, $p_i=p_{i,k}$ and $K=2b_m^2$, the first bullet point in Lemma \ref{technical appendix} is satisfied.

Let us verify the second one. The role of $l_i$ will be played by $D_{i,k}$. Note that, since $N_{i,k}$ is invertible modulo $b_i$, the number $\beta_{i,k}$ multiplied by $D_{i,k}$ generates $\Z/D_{i,k}\Z$. The same is true if we multiply $\beta_{i,k}$ by $q_k = \prod_{j} D_{j,k}$ as $\prod_{j\neq i} D_{j,k}$ is invertible modulo $D_{i,k}$. Note that $\beta_{i,k} \cdot q_k = p_{i,k}$. Finally, by the Chinese Remainder theorem, $\Z/q_k \Z$ admits a decomposition $\Z/q_k \Z \cong \bigoplus_i \Z/D_{i,k}\Z$, where $\Z/D_{i,k}\Z$ is generated by the image of $p_{i,k} \in \Z/q_k \Z$.

By Lemma \ref{technical appendix}, for the action $\Z^m\acts_\alpha \T^1$, the space $(t\T^1,d_{\Z^m})$ with
$t=lq=2^{m^{2k}}\prod_i D_{i,k}$
is quasi-isometric to $(\T^1, 2^{m^{2k}} d)\times \prod_i (\T^1, D_{i,k} d)$. But since $2^{m^{2k}}/b_i \leq D_{i,k} \leq 2^{m^{2k}}$, the latter is bi-Lipschitz equivalent to $(\T^{m+1}, 2^{m^{2k}}d)$.
\end{proof}

\end{document}